%% file: Barabasz_arxiv.tex
\documentclass{amsart}

\usepackage{amsaddr}
\usepackage{amsmath}
\usepackage{amsmath,mathtools}
\usepackage{amsfonts}
\usepackage{booktabs} % For formal tables
\usepackage{subfig}
\usepackage[ruled, linesnumbered, noend]{algorithm2e} % For algorithms
\usepackage{multicol}
\usepackage{xcolor}
\usepackage[utf8]{inputenc}
\usepackage[rightcaption]{sidecap}
\usepackage{colortbl}
\usepackage{soul}
\usepackage{graphicx}

\SetAlFnt{\small}
\SetAlCapFnt{\small}
\SetAlCapNameFnt{\small}
\SetAlCapHSkip{0pt}
\IncMargin{-\parindent}
\newtheorem{theorem}{Theorem}
\newtheorem{corollary}{Corollary}
\newtheorem{definition}{Definition}

\calclayout

	\title[Error Analysis and Improving the Accuracy of Winograd convolution for DNN]{Error Analysis and Improving the Accuracy of Winograd
		Convolution for Deep Neural Networks}

\author{Barbara Barabasz$^1$, Andrew Anderson$^1$, Kirk M. Soodhalter$^2$ and David Gregg$^1$}

\address{$^1$
	School of Computer Science and Statistics,
	Trinity College Dublin,
	Dublin 2,
	Ireland}

\address{$^2$
	School of Mathematics,
	Trinity College Dublin,
	Dublin 2,
	Ireland}

\email{barabasb@tcd.ie, aanderso@tcd.ie,ksoodha@maths.tcd.ie,dgregg@tcd.ie}
	
\begin{document}	
\maketitle
\begin{abstract}
	
	Popular deep neural networks (DNNs) spend the majority of their
	execution time computing convolutions. The Winograd family of
	algorithms can greatly reduce the number of
	arithmetic operations required and is present in many DNN software
	frameworks. However, the performance gain is at the expense of a
	reduction in floating point (FP) numerical accuracy. In this paper, we
	analyse the worst case FP error and prove the estimation of norm and conditioning of the algorithm. We show that the bound grows exponentially with
	the size of the convolution, but the error bound of the \textit{modified} algorithm is smaller than the original one. 
	We propose several methods for reducing
	FP error. We propose a canonical evaluation
	ordering based on Huffman coding that reduces summation error. We study the
	selection of sampling ``points'' experimentally and find empirically
	good points for the most important sizes. We identify the main factors
	associated with good points.  In addition, we explore other methods to
	reduce FP error, including mixed-precision convolution, and pairwise
	summation across DNN channels.  Using our methods we can significantly
	reduce FP error for a given block size, which allows larger block
	sizes and reduced computation.

\end{abstract}

\section{Motivation}
\label{sec:motivation}

Deep Neural Networks (DNNs) have become powerful tools for image, video, speech and language
processing. However, DNNs are very computationally demanding, both during and after training. A large part of this computations consists of convolution operations, which are used across a variety of DNNs, and in particular convolutional neural networks (CNNs). As DNNs become ubiquitous, reducing the cost of DNN convolution is increasingly important.
%However, DNNs are computationally demanding because
%they perform large numbers of small convolutions.

Simple \textit{direct} convolution require $O(n^2)$ operations to
convolve a size $n$ input with size $n$ convolution
\textit{kernel}. In contrast, \emph{fast} convolution algorithms
require asymptotically fewer operations. For example, converting the
kernel and input to the Fourier domain with the \textit{fast Fourier
	transform} (FFT) requires just $O(n log_2(n))$ operations. In the
Fourier domain, convolution can be computed in $O(n)$ operations by
pairwise multiplication (Hadamard product) of the input vectors.

Although FFT convolution is a popular approach, within the area of
deep neural networks (DNNs) a less well known algorithm is widely
used. The Winograd family of fast convolution algorithms attempts to
minimize the number of operations needed for \emph{fixed-size small
	convolutions}. Around 1980, Winograd proved that a convolution of the
input of length $n$ with a kernel of length $n_h$ can be computed
using a theoretical minimum of just $n+n_h-1$ \textit{general
	multiplications} (Hadamard product operations) \cite{Winograd80b}.

Winograd convolutions use a predetermined triple of linear
transforms. The first two transform the input and the kernel to space
where, like in the Fourier domain, pointwise multiplication (Hadamard
product) can be used to perform convolution. The third transform moves
the result back to the space of the inputs.  In Winograd convolution,
each of the transforms requires $O(n^2)$ operations, as compared to
$O(n log_2(n))$ for FFT. Thus, Winograd convolution is efficient only
for very small convolutions, or where the cost of the transform can be
amortized over many uses of the transformed data. In DNN convolution,
the kernels are small, typically $3 \times 3$ or $5 \times 5$, and
large inputs can be broken into a sequence of smaller
segments. Further, each input is convolved with many kernels, and each
kernel with many input segments, so the cost of transform operations
is amortized over multiple uses. Further, DNNs often operate on a
\textit{mini-batch} of many inputs (typically 32--512) at a time,
which further increases the re-use of each transformed kernel.
%that each transformed vector is used many times.

Although the transforms are expensive, Winograd convolution can
guarantee the theoretical minimum number of \textit{general
	multiplications}. The Winograd transform of a real-valued input is
real-valued, so that real (not complex) multiplication is used for
pairwise multiplication (Hadamard product).  Real-valued
multiplication requires just one machine multiply, whereas complex
multiplication requires four multiplies and two adds, or three
multiplies and five adds \cite{Knuth98}. Compared with the FFT
approach, Winograd convolution allows for faster Hadamard product
computations, at the cost of more expensive transforms.

Winograd convolution has a further weakness. The linear transforms are
pathologically bad cases for FP accuracy, as we describe in Section
\ref{sec:error-tc}. To obtain a good level of numerical
accuracy it is normally necessary to break convolutions with a large
input into a sequence of smaller ones. However, recall that Winograd
convolution requires $n+h_h-1$ \textit{general multiplications} to
convolve an input of size $n$ with a kernel of size $n_h$. Thus, when
a large input is split into segments, there is an overhead of $n_h-1$
additional \text{general multiplications} for each segment.

In this paper, we address the question of numerical accuracy of
Winograd convolution for deep neural networks. Better numerical
accuracy allows inputs to be split into a smaller number of
larger segments, which in turn reduces the number of \textit{general
	multiplications}.  We take both analytical and experimental
approaches to the problem. We isolate the components of error which
can be identified with mathematical analysis, and establish empirical
bounds for the components which cannot. We make the following specific
contributions.

 \begin{SCfigure}[][h]
	 \includegraphics[scale=0.4]{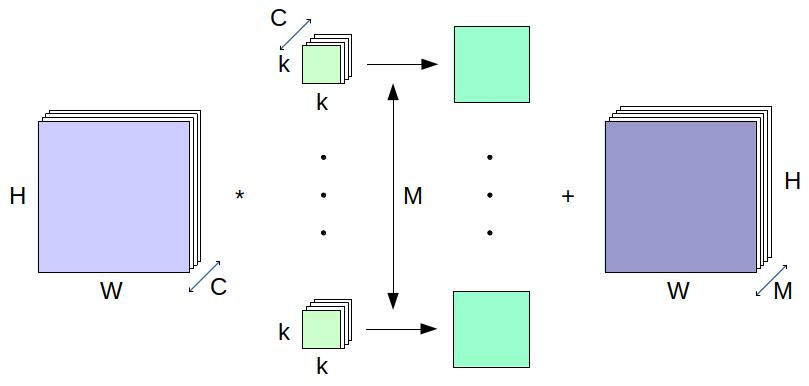}%{dnn-convolution-diagram}

	 \caption{Diagram of the structure of DNN convolution. 2D convolutions are
		 computed between kernels of size $k \times k$ and inputs of size $H \times
		 W$. Bundles of $C$ such convolution are summed pointwise to produce 1 output.
		 $M$ outputs are produced this way, and concatenated along the $M$ index to
		 produce the input for the next operation.}
	
	 \label{fig:DNN-diagram}
	 \end{SCfigure}

\vspace{-0.1cm}
\begin{itemize}
	
	\item We formalize and prove worst-case FP error bounds and
	identify error terms for the Toom-Cook algorithm, and show that error
	grows at least exponentially.
	
	\item We present a formal analysis of the error bounds for the
	``modified'' Toom-Cook algorithm, and prove that it has a lower,
	but nonetheless exponentially-growing error.
	
	\item We estimate the algorithm norm and conditioning.
	%\item \textcolor{red}{We analyse $1$ and $2$ dimension Winograd conditioning ?}
	
	\item We demonstrate that order of evaluation of FP
	expressions in the linear transform impacts accuracy. We propose a
	canonical Huffman tree evaluation order that reduces average error
	at no additional cost in computation.
	
	\item We experimentally evaluate strategies for selecting
	algorithm coefficients for typical DNN convolutions. We show
	relationships between coefficients which improve accuracy.
	
	\item We investigate algorithms that use a higher (double) precision transform.
	These methods reduce the error typically by around one third in our experiments.

\end{itemize}
\section{Fast Convolution and Deep Neural Networks}

In DNN convolution, each input segment (or kernel) is typically
convolved with many kernels (or input segments). When the 2D
convolutions are implemented with a fast convolution algorithm,
the transforms are thus amortized over many reuses of the transformed
segments\footnote{Note that once a DNN has been fully-trained, its
	weight kernels become constant.  The kernels can be stored
	pre-transformed when using the trained network. During DNN training,
	the kernel is updated on each training iteration, so the transform
	of the convolution must be computed for each
	convolution. Nonetheless, the transformed kernels and input segments
	can be reused many times \textit{within} each DNN
	convolution.}. Therefore, we are primarily interested in convolution
algorithms that minimize the \textit{general multiplications} which
implement the pairwise multiplication (Hadamard product) step.

Winograd \cite{Winograd80b} proved that the minimum number of general
multiplications (that is the multiplications used in the Hadamard product) 
is $n + n_h - 1$. Winograd demonstrated that
the existing Toom-Cook method \cite{Toom63,Cook66} is capable of
generating optimal convolution algorithms that achieve this minimum.
Winograd also developed his own method for generating fast algorithms.

In 2016 Lavin and Gray \cite{Lavin16} demonstrated that Winograd
convolution can be around twice as fast as \emph{direct} convolution
in DNNs \cite{Lavin16}.  A key contribution of their paper is an
algorithm to break multi-channel multi-kernel DNN convolution
into smaller segments that can be computed with
matrix multiplication. Lavin and Gray actually used the
Toom-Cook rather than Winograd's method to generate their convolution
algorithms\footnote{See Lavin and Gray's source code at:
	http://github.com/andravin/wincnn}. However, they described their
approach as ``Winograd convolution'' and within the DNN research
literature that term has come to include both Toom-Cook and Winograd
methods.

\begin{SCfigure}[][h]
	\includegraphics[scale=0.3]{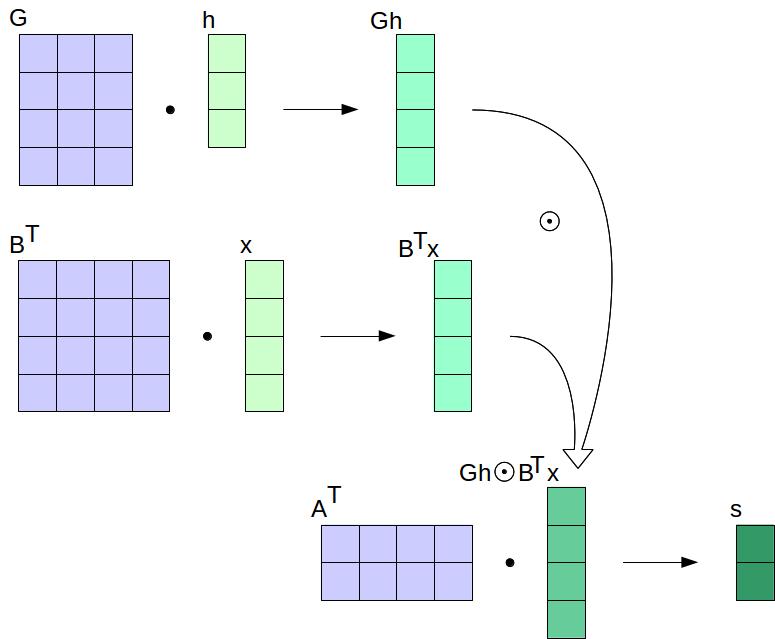}
	\hspace{1cm}
	
	\caption{Toom-Cook Convolution. The forward linear transform of
		kernel $h$ is computed by multiplication with matrix $G$. Similarly, the
		forward linear transform of input $x$ is computed by multiplication with a
		matrix $B^T$. Corresponding elements of the two resulting vectors are then
		multiplied pairwise (compute Hadamard product) to perform the convolution.
		The resulting vector is the \emph{image under the linear transform} of the actual
		convolution output. To translate the result back to the domain of the inputs,
		$h$ and $x$, the backward transform is applied by multiplication with a matrix
		$A^T$ to produce the final output $s$.}
	
	\label{fig:Toom-Cook}
\end{SCfigure}

\subsection{Decomposing Convolution and Minimizing General Multiplications}

A convolution of any output size $n_o > 1$ can be decomposed into the
sum of smaller convolutions. For example, a convolution with $n_o = 8$
can be computed as eight convolutions with $n_o=1$ (i.e. direct
convolution), four convolutions with $n_o=2$, two convolutions with
$n_o=4$, or one convolution with $n_o=8$.  With a kernel of size
$n_h=3$, the total number of general multiplications for each of these
decompositions will be $8 \times 3 = 24$, $4 \times 4 = 16$, $2 \times
6 = 12$ or $1 \times 10 = 10$ respectively.

The larger the size of each sub-convolution, the fewer general
multiplications are needed to compute the total output. Unfortunately,
bigger output sizes lead to larger FP errors. In fact, as we show in
Section \ref{sec:error-tc}, the error grows at least exponentially
with $n_o+n_h-1$.

Table \ref{tab:mult} summarizes the number of general multiplications
\emph{per output point} for different output block sizes using a
selection of typical kernel sizes from real-world DNNs. Clearly, we
would like to benefit from the efficiency of large output block
sizes. For example, a $5 \times 5$ Toom-Cook convolution with an
output block size of $12 \times 12$ (bottom right) uses around
$14\times$ fewer \textit{general multiplications} per output point
than direct convolution (top right). However, the error for $n_h = 5
\times 5$ is so large, that in practice sizes $n_o > 3 \times 3$ are
not used in any current DNN software framework.

\begin{table}[h]
	\caption{Number of multiplications per single output point for direct and Toom-Cook convolutions for different kernel and output size}
	\label{tab:mult}
	\begin{minipage}{\columnwidth}
		\begin{small}
			\begin{center}
				\begin{tabular}{l|cc|cc|cc|cc}
					\toprule
					No of & Output & Mult/ & Output size & Mult/ & Output & Mult/ & Output & Mult/ \\
					points & $K=3$ & output & $K=3 \times 3$ & output & for $K=5$ & output & $K=5 \times 5$ & output\\
					\hline 
					0     & 1   & 3    & 1$\times$1   & 9    & 1 & 5   & 1$\times$1  & 25\\
					4     & 2   & 2    & 2$\times$2   & 4    & - & -   & -  & -\\
					5     & 3   & 1.67 & 3$\times$3   & 2.78 & - & -   & -  & -  \\
					6     & 4   & 1.5  & 4$\times$4   & 2.25 & 2 & 3   & 2$\times$2  & 9  \\
					7     & 5   & 1.4  & 5$\times$5   & 1.96 & 3 & 2.33& 3$\times$3  & 5.44 \\
					8     & 6   & 1.34 & 6$\times$6   & 1.78 & 4 & 2   & 4$\times$4  & 4\\
					9     & 7   & 1.29 & 7$\times$7   & 1.65 & 5 & 1.8 & 5$\times$5  & 3.24 \\
					10    & 8   & 1.25 & 8$\times$8   & 1.56 & 6 & 1.67& 6$\times$6  & 2.78 \\
					11    & 9   & 1.22 & 9$\times$9   & 1.49 & 7 & 1.57& 7$\times$7  & 2.47 \\
					12    & 10  & 1.2  & 10$\times$10 & 1.44 & 8 & 1.5 & 8$\times$8  & 2.25 \\
					13    & 11  & 1.18 & 11$\times$11 & 1.4  & 9 & 1.44& 9$\times$9  & 2.09 \\
					14    & 12  & 1.17 & 12$\times$12 & 1.36 & 10& 1.4 & 10$\times$10& 1.96\\
					15    & 13  & 1.15 & 13$\times$13 & 1.33 & 11& 1.36& 11$\times$11& 1.86 \\
					16    & 14  & 1.14 & 14$\times$14 & 1.31 & 12& 1.33& 12$\times$12& 1.78
				\end{tabular}
			\end{center}
			\bigskip\centering
		\end{small}
	\end{minipage}
\end{table}

\section{Toom-Cook algorithm}
\label{sec:Toom-Cook}
In this section we describe the Toom-Cook convolution algorithm. It is based on the Chinese Remainder Theorem (CRT) for polynomials and the Matrix Exchange Theorem. Toom \cite{Toom63} and Cook \cite{Cook66} provide details on the theoretical background. Parhi \cite{Parhi07}, Tolimieri \cite{Tolimeri97} and Blahut \cite{Blahut10} provide useful descriptions of using the Toom-Cook algorithm to perform a discrete convolution.
The one-dimensional discrete convolution of two vectors $h\begin{bmatrix}h_1 & \cdots & h_{n_h}\end{bmatrix}$ and $x=\begin{bmatrix}x_1 & \cdots & x_n\end{bmatrix}$ is the vector $s=h \ast x$ where $s_i = \sum_{j=1}^ih_ix_{i-j}$.

The main idea of Toom-Cook convolution is to transform the kernel and input into the \emph{modulo polynomial} domain where convolution becomes an element-wise multiplication (Hadamard product) and then transform the result back. We construct three matrices, one for the kernel transform, one for the input transform, and one to transform the result back. We denote these matrices as $G$, $A$ and $B$ respectively.

\begin{theorem}[Chinese Remainder Theorem for polynomials]
	\label{thm.chinese-remainder}
	Let $\mathbb{F}[a]$ be the ring of all polynomials over a field $\mathbb{F}$.
	Consider the polynomial $M(a) \in \mathbb{F}[a]$  such as $M(a)=m_1(a)...m_{\ell}(a)$ where $m_i(a)$ irreducible $\forall i=1,2,..,\ell$ and
	$GCD(m_i(a),m_j(a))=1$ $\forall i=1,2,..,\ell, \ i \neq j$. Let $s_i(a) \in \mathbb{F}[a]$ be any polynomials for $i=1,2,..,\ell$.
	Then there exists $s(a) \in \mathbb{F}[a]$ a unique solution of the system of congruences: 
	$$s(a) = s_i(a) \; mod \; m_i(a)\qquad \forall\ i =1,2,..,\ell$$
	and
	$$s(a) = \sum_{i=1}^{\ell} s_i(a)N_i(a)M_i(a) \; mod \; M(a)$$
	where:
	$$N_i(a)M_i(a)+n_i(a)m_i(a)=1,\qquad N_i(a), \; M_i(a) = \frac{M(a)}{m_i(a)},\ \mbox{and}\qquad m_i(a) \in \mathbb{F}[a].$$ 
\end{theorem}

Let us assume in what follows that $\mathbb{F}$ is equal to $\mathbb{R}$, a field of real numbers, and represent the one-dimensional kernel vector 
$h=\begin{bmatrix} h_1 & h_{2} & \cdots  &h_{n_h} \end{bmatrix}$ 
and one-dimensional input vector 
$x=\begin{bmatrix} x_1 & x_{2} & \cdots  &x_{n} \end{bmatrix}$ 
as polynomials $h(a)$ and $x(a)$, with coefficients equal to their respective components, such that the leading coefficients of $h(a)$ and $x(a)$ are taken to be $h_{n_h}$  and $x_n$, respectively. Then computing the one-dimensional discrete convolution is equivalent to computing the coefficients of the polynomial product \linebreak $s(a)=h(a)x(a)$.

In the Toom-Cook algorithms, it is assumed that all $m^{(i)}(a)$ are monomials; so the computation reduces to the following steps:
\begin{enumerate}
	\item Choosing points $p_i$ to construct polynomials $m_i(a)=a-p_i$;
	\item Evaluating polynomials $h(a)$, $x(a)$ at each point $p_i$ to change the domain, which is equivalent to computing $h_i(a)=h(a) \; mod \; m_i(a)$ and $x_i(a)=x(a) \; mod \; m_i(a)$;
	\item Performing the multiplication $s_i(a)=h_i(a)x_i(a)$; and
	\item Applying the Chinese Reminder Theorem to compute the coefficients of the polynomial $s_i$.
\end{enumerate}
We can represent this algorithm as:
$$V^{-1}(V_x x \odot V_h h)$$
where matrices $V_x$ and $V_h$ (called Vandermonde matrices) represent transformation into the \emph{modulo polynomial} domain, which is equivalent to evaluation of the polynomial in different points $p_i$. The matrix $V^{-1}$ is the inverse Vandermonde matrix for the transformation of the result back from the \emph{modulo polynomial} domain. The nonsingularity of these matrices is guaranteed by choosing the different $a_i$ so that the assumptions of the CRT are fulfilled.

\subsection{Matrix Interchange}

It is possible to interchange matrices in the convolution formula using the Matrix Exchange theorem. \\
\begin{theorem}
	Let $M$ be a diagonal matrix. If matrix $M$ can be factorised as \linebreak $M=CDE$ then it also can be factorised as $M=(\overline{E})^TD(\underline{C})^T$, where matrix $\overline{E}$ is a matrix obtained from $E$ by reversing the order of its columns and $\underline{C}$ is a matrix obtained from $C$ by reversing its rows.
\end{theorem}

Although the literature on DNNs typically calls this operation convolution, from a mathematical point of view the operation we want to compute is, in fact, the \emph{correlation}. This is why, when applying the Matrix Exchange Theorem, we do not reverse the order of columns in matrix $E$. Thus

$$V^{-1}(V_xx \odot V_hh) = V^{-1}Diag(V_hh)V_xx={V_x}^T(V_hh \odot V^{-T}x)$$
Putting $A=V_x$, $G=V_h$ and $B=V^{-1}$ we obtain the following formula for one-dimensional convolution
$$(h \ast x)_{1D}=A^T(Gh \odot B^Tx)$$
In a similar way, using the Kronecker product, we obtain a formula for two-dimensional convolution
$$(H \ast X)_{2D}=A^T(GHG^T \odot B^TXB)A$$
where matrices $H$ and $X$ are the two-dimensional kernel and input, respectively. \\

\subsection{Linear Transform Matrix Construction}
\label{subsec:Matrices_constr}

The method of constructing matrices $A^T$, $G$ and $B^T$ is presented in Algorithm \ref{alg:Toom-Cook}.
To compute a 1D convolution of size $n_o$ with the kernel of size $n_h$, we need a input of size $n=n_h+n_o-1$. As inputs to the algorithm we provide $n$ different real points $p_1,..,p_{n}$ and use them to construct $n$ linear polynomials $m_i(a)=x-p_i$, for $i=1,..,n$. We compute polynomial $M(a)=(a-p_1)..(a-p_n)$ and polynomials $M_i(a)=M(a)/m_i(a)=\sum_j{M_{j,i}a^{j-1}}$ for $i=1,2,\ldots,n$ and $j=0,1,\ldots,n$ used in CRT.

The matrix $A^T$ is a transposed rectangular Vandermonde matrix of size $n_o \times n$. We compute its elements as the $0$th to $n_o-1$th powers of the $n$ selected points.
Next, we construct the matrix $G$ of size $n \times n_h$ in a very similar way.
%Note that we must apply a scaling factor, $N^i$, to the outputs of the pairwise multiplication %stage of Toom-Cook convolution. However, immediately afterward we multiply this output by the %matrix $G$. Rather than applying the pointwise scaling factors $N^i$ and then applying $G$, we %instead embed the scaling factors into the matrix $G$.
Note that we scale one of the Vandermonde matrices by coefficients $N_i$ to obtain matrices $G$ and $B^T$.
We find the coefficients $N_i$ using the Euclidean algorithm \cite{Biggs02}.

{\fontsize{8}{10}
	\begin{algorithm}[h]%[t] 
		\begin{multicols}{2}
			\SetAlgoNoLine
			\KwIn{$n_o$ - size of output, $n_h$ - size of kernel, \; $\{p_1,\cdots,p_n\}$ set of $n$ different points}
			\KwOut{Three matrices $A^T$, $G$ $B^T$ for Toom-Cook convolution}
			$n=n_o+n_h-1$\;
			%    $M_i = symbolic(\prod_{k \neq i}{a-p_k})$ \;
			\For{$i=1$ to $n$}{
				\For{$j=1$ to $n$}{
					$M_{i,j}$ = coefficient of the polynomial $\prod_{k \neq i}({a-p_k})$ stands for $a^{j-1}$
				}
			}
			\For{$i=1$ to $n$}{
				$N_i=\frac{1}{\prod_{j \neq i}(p_i-p_j)}$ }
			\For{$i=1$ to $n_o$}{
				\For{$j=1$ to $n$}
				{
					$A^T_{i,j}={p_j}^{i-1}$
				}
			}
			%//macierz $G$
			\For{$i=1$ to $n$}{
				\For{$j=1$ to $n_h$}{
					$G_{i,j}={p_{i}}^{j-1}*N_i$
				}
			}
			%//macierz $B^T$
			\For{$i=1$ to $n$)}{
				\For{$j=1$ to $n$}{
					$B^T_{i,j}=M_{j,i}$
				}
			}
		\end{multicols}
		\caption{Toom-Cook algorithm}
		\label{alg:Toom-Cook}
	\end{algorithm}
}

The general form of matrices obtained by the Toom-Cook algorithm is as follows:\\

\footnotesize
$$
{\arraycolsep=2pt
	G =\begin{bmatrix}
	1 & p_1*N_1 & \cdots & {p_1}^{n_h-1}*N_1 \\
	1 & p_2*N_2 & \cdots& {p_2}^{n_h-1}*N_2 \\
	\vdots & \vdots & \ddots & \vdots \\
	1 & p_{n}*N_n & \cdots & {p_{n}}^{n_h-1}*N_n
	\end{bmatrix};\ \ 
	A^T =\begin{bmatrix}
	1 & 1 & \cdots & 1 \\
	p_1 & p_2 & \cdots& p_{n} \\
	\vdots & \vdots & \ddots & \vdots \\
	p_1^{n_o-1} & p_2^{n_o-1}  & \cdots & {p_{n}}^{n_o-1}
	\end{bmatrix}; \mbox{ and}\ \ \  
	B^T=\begin{bmatrix}
	M_{1,1} & \cdots & M_{1,n} \\
	\vdots & \ddots & \vdots\\
	M_{n,1} & \cdots & M_{n,n}
	\end{bmatrix}.
}
$$\normalsize

\subsubsection*{A note on matrix construction}
Theoretically, the evaluation of the polynomials at the chosen
interpolation points corresponds to the action of square Vandermonde
matrices on the coefficient vectors.  These matrices are nonsingular
due to our choice of points $p_{i}$.  Thus, in our analysis, we use
properties of square Vandermonde matrices to understand the stability
properties of the Toom-Cook algorithm and conditioning of the
underlying calculation.  The properties of \textit{square} Vandermonde
matrices are well understood \cite{Pan16}, but the matrices $G$ and
$A$ as described in our implementation are actually rectangular.
However, this is an advantage we take at the algorithmic level rather
than a mathematical property of the interpolation process.  We can
mathematically interpret the actions of $G$ and $A$ as square
Vandermonde matrices acting on vectors whose last entry is zero.
Thus, we can analyse these methods in terms of the square matrices
while the implementation is done in terms of rectangular matrices
which are the square matrices with the last column deleted.

The matrix $G$ shown in Figure \ref{fig:Toom-Cook} has only three
elements in each row, rather than four, because the kernel in the
example has just $n_h=3$ elements. The full (square) Vandermonde
matrix $G$ actually has four elements per row, with the fourth element
computed in the same pattern as the first three. The kernel $h$ also
has four elements, but the fourth is always zero. Thus, the fourth
element of each row of $G$ is multiplied by the fourth element of $h$
which is always zero. As a result, we can safely eliminate the last
column of the square Vandermonde matrix $G$, and crucially, \emph{all
	associated computation}.

Similarly, $A^T$ in Figure \ref{fig:Toom-Cook} is shown with just two rows
rather than four, because in this example we compute an output block of size
two (that is equivalent to the number of \emph{fully computed} elements).
However, we could equally show all four rows of the Vandermonde matrix $A$ and
discard two of the computed results.

\section{Error in Toom-Cook Convolution}
\label{sec:error-tc}
In this section, we derive a bound on the FP error that can
arise in Toom-Cook convolution. We use the methods and notation of
Higham's standard textbook on FP error \cite{Higham02}. In line with
Higham, and many other formal analyses of FP error
\cite[p. 48]{Higham02}, we provide a worst-case error analysis.

\subsection{FP error}

FP error arises because FP numbers are limited precision
approximations of real numbers. Each real number can be mapped to its
nearest FP equivalent with the rounding function $x
\rightarrow fl(x)$ such that $fl(x)=min_{f \in F}(f-x)$. Where the
absolute value of a real number is larger than the largest
representable FP number, the number is said to overflow. Overflow
results in a catastropic loss of
accuracy, but it is rare at least within the field of DNNs. In the
absence of overflow, we have the assumption $fl(x)=x(1+\delta)$ where $-\varepsilon < \delta < \varepsilon$ and $\varepsilon$ 
is a machine epsilon dependent on
precision.  Similarly provided there is no overflow in inputs or
results, FP arithmetic operators can be described as follows: $fl(x \;
\mathrm{op} \; y)=(x \; \mathrm{op} \; y)(1+\delta)$, where $|\delta|
\leq \varepsilon$, $x,y \in F$ and $\mathrm{op}\in\left\lbrace +,-,*,/\right\rbrace$; see, e.g., 
\cite{Higham02,Goldberg91,Wilkinson94}.
In this paper, for a quantity $x$, we denote the floating point representation of $x$ 
by $\hat{x}$ and FP operations by $fl(\cdot)$.

\subsection{FP Error in the Linear Transforms}
\label{subsec:LinTransf}

The core operation in the linear transforms is a
matrix-vector product, which can be represented as a set of dot products
$a^Tx$.  Let us take an input vector $x=\begin{bmatrix} x_1 & x_{2} & \cdots & x_n \end{bmatrix}^{T} $ where $x_i \in
F$, $\forall i=1,2,\ldots,n$, and another vector $a= \begin{bmatrix} a_1 & a_{2} &\cdots  & a_n \end{bmatrix}^{T} $ which is
part of the algorithm, so $fl(a)= \begin{bmatrix} fl(a_1) & fl(a_2) & \cdots & fl(a_n) \end{bmatrix}^{T} $ and \linebreak
$fl(a_i)=a_i(1+\delta_i)$, where $|\delta_i| \leq \varepsilon$
$\forall a_i=1,2,\ldots,n$ \cite{Higham02}.  Then

\begin{equation}
\label{eq:dot_prod}
|a^Tx - fl(fl(a^T)x)| \leq |a^T||x|\alpha^{(n)}\varepsilon+O(\varepsilon^2)
\end{equation}

Higham provides a similar bound on the error for dot product but uses $n$ where we use $\alpha^{(n)}$.
That is because he assumes linear summation in dot product computations.
There is a wide range of summation methods that allows us to compute dot product with smaller floating point error than using linear summation. Demmel and Rump have analysed various summarion algorithms. The algorithms, as well as their floating point  error estimations, can be found in (\cite{Rump08I}, \cite{Rump08II}, \cite{Demmel04}).
Fo generality, we do not assume any particular method of
dot product evaluation. Instead, we use $\alpha^{(n)}$,
which stands for the error of dot product computations for vectors of $n$ elements.

Also in our analysis, the vector $a^T$ is a constant, not an input. The value of $a^T$ depends on the parameters of the algorithm. We write $fl(a^T)$ because the mathematically exact value of $a^T$ may not be exactly representable in finite precision FP. We want to estimate the error of the algorithm, as it depends on these parameters, as well as of the number and type of operations.

Note that the value of $\alpha^{(n)}$ depends on the error from multiplication, as well as on $n$ and on the method of summation.
We have three possible cases that give us different boundaries for the error of multiplication $a_ix_i$:
\begin{itemize}
	\item Values of $a_i$ are not exactly representable in $F$ (the set of FP numbers). In this case we have an error from the inexact representation of $a_i$ and from the multiplication, so $fl(a_i)=a_i(1+\delta_i)$ where $|\delta_i| \leq \varepsilon$. Then $|fl(fl(a_i)x_i)-a_ix_i| \leq |a_i||x_i|2\varepsilon+O(\varepsilon^2)$, $\forall i=1,2,\ldots,n$.
	\item Values of $a_i$ are exactly representable in $F$. In this case only the multiplication and summation errors remain, that is: $|fl(fl(a_i)x_i)-a_ix_i| \leq |a_i||x_i|\varepsilon+O(\varepsilon^2)$, $\forall i=1,2,\ldots,n$.
	\item When of $a_i$ are integer powers of $2$ we have no error from either representation or from multiplication, so $|fl(fl(a_i)x_i)-a_ix_i| \leq |a_i||x_i|$, $\forall i=1,2,\ldots,n$.
\end{itemize}

If we assume \emph{linear} summation in Equation~\ref{eq:dot_prod} we have
$\alpha^{(n)}=n+1$ for any elements $a_i$, $\alpha^{(n)}=n$ for $a_i$ exactly
represented in $F$ and $\alpha^{(n)}=n-1$ if all $a_i$ are integer powers of $2$. However $n-1 < n < n+1$ so using $\alpha^{(n)}=n+1$ is a correct estimate but does not give the tightest possible bound.\\

\subsection{Toom-Cook Convolution Error Estimation}
\label{sec:Toom-Cook-error}

In this section, we present a formal error analysis of the Toom-Cook convolution algorithm, which to our knowledge is the first such formulation. Our approach uses the Higham \cite{Higham02}  method of FP error estimation and results on the instability of Vandermonde systems by Higham \cite{Higham02} and Pan \cite{Pan16}. The error estimation allows us to show that the Toom-Cook convolution algorithm is unstable and to identify the components of the error.

The Toom-Cook method generates algorithms for fixed-size convolution,
which are expressed as a set of three matrices, $G$, $B^T$ and $A^T$.
These matrices are computed once, ahead of time, and can be used with
many different inputs and kernels.  Figure \ref{fig:Toom-Cook} shows
the three steps of the algorithm: (a) linear transforms of the kernel,
$h$, and input, $x$; (b) pairwise multiplication between the elements
of the transformed input and kernel (Hadamard product); and (c) the output
linear transform. All of these operations have an impact on the
accuracy of the result, so we see terms in our error for each
operation.

To estimate the error bounds we will use the matrix norm $\|A\|_1=max_j \sum_i|A_{ij}|$ that is induced by the vector norm $\|y\|_1=\sum_i |y_i|$, and also the matrix norm \linebreak $\|A\|_F=(\sum_i \sum_j |a_{ij}|^2)^{1/2}$ (called the Euclidean or Frobenius norm), which is not an induced norm.   However, it is equivalent to the
matrix norm $\|A\|_2$ induced by vector norm $\|y\|_2=(\sum_i|y_i|^2)^{1/2}$ with the inequality 
\begin{equation}\label{eqn.Frob2NormBounds}
\|A\|_2\leq \|A\|_F \leq \sqrt{r}\|A\|_2,
\end{equation}
where $r$ is the
rank of $A$.
The Euclidean norm is used very often in numerical analysis instead of $\|\cdot\|_2$, because it is easier to compute and  $\|A\|_F=\|\,|A|\,\|_{F}$, where $|A|$ is matrix with entry-wise absolute values of the entries of $A$ \cite{Wilkinson94}. We define $\alpha^{(n)}$, $\beta^{(n)}$ and $\gamma^{(n_h)}$ as constants used in dot product FP error bounds in Equation~\ref{eq:dot_prod} for matrices $A^T$, $B^T$ and $G$.

\begin{theorem}
	\label{theorem:TC-1D-error}
	The error for one-dimensional Toom-Cook convolution computation satisfies the normwise bound equal to:
	\begin{equation}
	\label{eq:TC-1D-norm}
	\|\hat{s} - s\|_1 \leq \|A^T\|_1\thinspace\|G\|_F\thinspace\|h\|_2\thinspace\|B^T\|_F\thinspace\|x\|_2\thinspace\left(\alpha^{(n)}+\beta^{(n)}+\gamma^{(n_h)}+1\right)\thinspace\varepsilon+O(\varepsilon^2)
	\end{equation}
	
	Error for the $q$th element of one-dimensional Toom-Cook convolution computation satisfies the bound equal to:
	\begin{equation}
	\label{eq:TC-1D-element}
	|{\hat{s}}_q-s_q| \leq |A^T|\left(|G||h|\odot|B^T||x|\right)\left(\alpha^{(n)}+\beta^{(n)}+\gamma^{(n_h)}+1\right)\varepsilon + O(\varepsilon^2)
	\end{equation}
	Where values of $\alpha^{(n)}$, $\beta^{(n)}$ and $\gamma^{(n_h)}$ depends on method of summation in dot product computations in matrices $A^T$, $B^T$ and $G$ as in formula (\ref{eq:dot_prod})

\end{theorem}

\begin{proof}
	Let $f(h,x)$ be the bilinear function computing Toom-Cook convolution \\ $f(h,x)$:${\mathbb{R}}^{n_h} \times {\mathbb{R}}^{n}\rightarrow \mathbb{R}^{n_o}$ such that $f(h,x)=A^T(Gh \odot B^Tx)$. 
	The computation consists of (a) kernel and input transformations, $f_1^{h}:{\mathbb{R}}^{n_h} \rightarrow {\mathbb{R}}^{n}$,  $f_1^{h}(h)=Gh$, and $f_1^{x}: {\mathbb{R}}^{n} \rightarrow {\mathbb{R}}^{n}$ $f_1^{x}(x)=B^Tx$; (b) Hadamard product: $f_2:{\mathbb{R}}^{n}\times {\mathbb{R}}^{n} \rightarrow {\mathbb{R}}^{n}$ $f_2(b,c)=b \odot c$; and \linebreak (c) postprocessing transformation $f_3:{\mathbb{R}}^{n} \rightarrow {\mathbb{R}}^{n_o}$ $f_3(a)=A^Ta$.
	
	We therefore need to find the error for the composition of these three computations, that is the error of $f(h,x)=f_3(f_2(f_1^{h}(h),f_1^{x}(x)))$. We follow Higham's method \cite{Higham02} for estimating the FP result of the composed function.
	
	Let $a_1=(h,x)$, and $a_{k+1}=f_k(a_k)$ that is the result of $k+1$th stage of
	the algorithm. So $a_2$ is the vector that includes preprocessing transforms of
	kernel and input \linebreak $f_1^{h}(h)=Gh$ and $f_1^{x}(x)=B^Tx$, $a_3$ is the Hadamard
	product of the two vectors $Gh$ and $B^Tx$ and is equal to $f_2(Gh, B^Tx)=Gh
	\odot B^Tx$ and $a_4$ is the postprocessing transform $f_3(Gh \odot
	B^Tx)=A^T(Gh \odot B^Tx)$. The computed values are denoted by
	$\hat{a}_{k+1}=f_k(a_k)+\Delta a_{k+1}$, so $\Delta a_{k+1}$ is the FP error of
	the $k$th stage of algorithm that we compute using formula \ref{eq:dot_prod}. \\
	%	\pagebreak
	Let vector $s=f_3(f_2(f_1^{h}(h),f_1^{x}(x)))$ be a real result and $\hat{s}$
	be the computed solution. \linebreak $J_k$ is the Jacobian matrix of $f_k$.
	%so $J_k=(\frac{\partial f_i}{\partial a_j})$.
	The computed result $\hat{s}$ is equal to the formula \cite{Higham02}:
	$$\hat{s}=f_3(f_2(f_1^{h}(h),f_1^{x}(x)))+J_3J_2\Delta a_2+J_3\Delta a_3 + \Delta a_4$$
	Where:
	\begin{align*}\nonumber 
	J_3=A^T,  & \hspace{10pt}  J_2=\left[Diag(B^Tx),Diag(Gh)\right] \\[5pt]
	|\Delta a_2| &\leq \left[
	\begin{array}{l}
	|G||h|\gamma^{(n_h)} \varepsilon +O(\varepsilon^2)\\
	|B^T||x|\beta^{(n)} \varepsilon + O(\varepsilon^2)
	\end{array}
	\right] \\[5pt]
	|\Delta a_3| \leq \left(|G||h| \odot |B^T||x|\right)\varepsilon+O(\varepsilon^2), &\hspace{10pt} |\Delta a_4| \leq |A^T|\left(|G||h| \odot |B^T||x|\right)\alpha^{(n)} \varepsilon +O(\varepsilon^2) 
	\end{align*}
	The componentwise error is the absolute difference between real and computed solutions \cite{Higham02} \cite{Wilkinson94}
	\begin{align*}\nonumber 
	|\hat{s}-s| &= \\
	=|f_3\left(f_2\left(f_1^{h}\left(h\right),f_1^{x}\left(x\right)\right)\right)+J_3J_2\Delta a_2 & + J_3 \Delta a_3 + \Delta a_4-f_3\left(f_2\left(f_1^{h}\left(h\right),f_1^{x}\left(x\right)\right)\right)|=\\[5pt]
	=|J_3J_2 \Delta a_2 + J_3 \Delta a_3+\Delta a_4| &\leq |J_3||J_2||\Delta a_2|+|J_3||\Delta a_3|+|\Delta a_4| \leq\\[5pt]
	\leq \left(|A^T||Diag\left(Gh\right)|,|A^T|\right.|Di&ag\left.\left(B^Tx\right)|\right)|\Delta a_2|+|A^T||\Delta a_3|+|\Delta a_4| \leq\\[5pt]
	\leq \left(|A^T||Diag\left(B^Tx\right)|\right., |A^T||&\left. Diag\left(Gh\right)|\right) \left[
	\begin{array}{l}
	|G||h|\gamma^{(n_h)} \varepsilon + O\left(\varepsilon^2\right)\\
	|B^T||x|\beta^{(n)} \varepsilon + O\left(\varepsilon^2\right)
	\end{array} 
	\right] + \\
	+ |A^T|\left(|G||h| \odot |B^T||x|\right)\varepsilon + O(&\varepsilon^2) +
	+|A^T|\left(|G||h| \odot |B^T||x|\right)\alpha^{(n)} \varepsilon + O\left(\varepsilon^2\right) =\\[5pt]
	=|A^T|\left(|G||h| \odot |B^T||x|\right)(\gamma&^{(n_h)}+\beta^{(n)})\varepsilon+|A^T|\left(|G||h|\odot|B^T||x|\right)\varepsilon + \\
	+ |A^T|\left(|G||h|\right. &\odot \left.|B^T||x|\right)\gamma^{(n_h)} \varepsilon + O\left(\varepsilon^2\right)=\\[5pt] 
	=|A^T|\left(|G||h| \odot |B^T||x|\right)&\left(\alpha^{(n)}+\beta^{(n)}+\gamma^{(n_h)}+1\right)\varepsilon + O\left(\varepsilon^2\right)
	\end{align*}
	For the normwise error estimation we use induced norm $\|\cdot\|_1$ and Frobenius norm $\|\cdot\|_F$, hence
	%  \begin{align*} 
	$$\|\hat{s}-s\|_1 \leq  $$
	$$  \leq \|A^T\left(Gh \odot B^Tx\right)\left(\alpha^{(n)}+\beta^{(n)}+\gamma^{(n_h)}+1\right)\varepsilon+O\left(\varepsilon^2\right)\|_1 \leq $$
	$$  \leq \|A^T\left(Gh \odot B^Tx\right)\|_1\left(\alpha^{(n)}+\beta^{(n)}+\gamma^{n_h}+1\right)\varepsilon +O(\varepsilon^2) \leq $$
	$$  \leq \|A^T\|_1\|Gh \odot B^Tx\|_1\left(\alpha^{(n)}+\beta^{(n)}+\gamma^{(n_h)}+1\right)\varepsilon +O(\varepsilon^2)$$
	%  \end{align*}
	Applying the Buniakowski-Schwartz inequality to componentwise multiplication
	yields
	$$\|\hat{s}-s\|_1 \leq \|A^T\|_1\thinspace\|Gh\|_2\thinspace\|B^Tx\|_2\left(\alpha^{(n)}+\beta^{(n)}+\gamma^{(n_h)}+1\right)\varepsilon +O(\varepsilon^2)$$
	Finally from norm equivalence \eqref{eqn.Frob2NormBounds} we have
	$$\|\hat{s}-s\|_1 \leq \|A^T\|_1\thinspace\|G\|_F\thinspace\|h\|_2\thinspace\|B^T\|_F\thinspace\|x\|_2\left(\alpha^{(n)}+\beta^{(n)}+\gamma^{(n_h)}+1\right)\varepsilon +O(\varepsilon^2)$$
\end{proof}

As $\|\cdot\|_1 \leq n\|\cdot\|_2$ we have
\begin{equation}
\label{eq:normE-1D}
\|\hat{s}-s\|_1 \leq n\|A^T\|_F\|G\|_F\|h\|_2\|B^T\|_F\|x\|_2\left(\alpha^{(n)}+\beta^{(n)}+\gamma^{(n_h)}\right)\varepsilon +O(\varepsilon^2)
\end{equation}

\begin{corollary}
	For linear summation in the dot product and for any elements in matrices $A^T$,
	$G$ and $B^T$ the componentwise boundary is equal to:
	$$|\hat{s} -s| \leq |A^T|\left(|G||h| \odot |B^T||x|\right)(n_h+2n+4)\varepsilon +O(\varepsilon^2)$$
	and the normwise boundary is equal to:
	$$\|\hat{s} - s\|_1 \leq \|A^T\|_1\thinspace\|G\|_F\thinspace\|h\|_2\thinspace\|B^T\|_F\thinspace\|x\|_2\left(n_h+2n+4\right)\varepsilon +O(\varepsilon^2)$$
	Where $n_h$ is the kernel size and $n$ is the input size of the convolution
\end{corollary}

\subsection{Two Dimensions}
\label{sec:2d-analysis}
Two-dimensional convolution can be implemented by nesting 1D convolutions \cite{Lavin16}. This nesting approach requires additional pre-/post-processing linear transforms.
For two-dimensional Toom-Cook convolution the analogous theorem is formulated as follows:
\begin{theorem}
	\label{theorem:TC-2D-error}
	Error for two-dimensional Toom-Cook convolution computation satisfies the componentwise bound equal to:
	\begin{equation}
	\label{eq:TC-2D-element}
	|\hat{S}-S| \leq |A^T|\left(|G||H||G^T| \odot |B^T||X||B|\right)|A|\left(2\alpha^{(n)}+2\beta^{(n)}+2\gamma^{(n_h)}+1\right)\varepsilon+\O(\varepsilon^2) 
	\end{equation}
	Error for two-dimensional Toom-Cook convolution computation satisfies the normwise bound equal to:
	\begin{equation}
	\label{eq:TC-2D-norm}
	\begin{split}
	\|\hat{S} - S\|_1 &\leq \\ 
	\leq \|A^T\|_1\thinspace\|G\|_F\thinspace\|H\|_F\thinspace\|G^T\|_F\thinspace\|B^T\|_F\thinspace\|X\|_F\thinspace\|B\|_F\thinspace&\|A\|_1\thinspace R \thinspace \varepsilon+O(\varepsilon^2) \\
	\mathrm{where:} \;\; R = 2\alpha^{(n)}+2\beta^{(n)}+2\gamma^{(n_h)}+1
	\end{split}
	\end{equation}
	We assume identical method of summation for matrix and transpose matrix
	multiplication, where $\alpha^{(n)}$, $\beta^{(n)}$, $\gamma^{(n_h)}$ represent
	errors from multiplication by matrices $A^T$, $B^T$ and $G$ respectively.
\end{theorem}

The proof of this theorem is presented in Appendix \ref{sec:2dim}

Notice that the Euclidean norm of any matrix $M$ is equal to the Euclidean norm of matrix $M^T$, so we can formulate the normwise boundaries
$$\|\hat{S}-S\|_1 \leq \|A^T\|_1\thinspace\|A\|_1\thinspace{\|G\|^2}_F\thinspace\|H\|_F\thinspace{\|B^T\|^2}_F\thinspace\|X\|_F\left(2\alpha^{(n)}+2\beta^{(n)}+2\gamma^{(n_h)}+1\right)\varepsilon+O(\varepsilon^2)$$

Notice that we can bound $\|A^T\|_1\|A\|_1 \leq n^2\|A^T\|_F\|A\|_F = n^2{\|A^T\|^2}_F$ \cite{Higham02}. Then we have the error bound estimation for two-dimensional Toom-Cook algorithm equal to $$\|\hat{S} -S\|_1 \leq n^2{\|A^T\|^2}_F\thinspace{\|G\|^2}_F\thinspace{\|B^T\|^2}_F\thinspace\|H\|_F\thinspace \|X\|_F\left(2\alpha^{(n)}+2\beta^{(n)}+2\gamma^{(n_h)}+1\right)\varepsilon+O(\varepsilon^2)$$
Comparing it to the one-dimensional Toom-Cook convolution \ref{eq:TC-1D-element}
we can observe that the error boundary for $2$ dimensions is approximately the square of the error of the $1D$ algorithm.

\begin{corollary}
	For linear summation in the dot product and for any elements in matrices $A^T$, $G$ and $B^T$ the componentwise boundary for two-dimensional Toom-Cook convolution is equal to:
	$$|\hat{S}-S| \leq |A^T|\left(|G||H||G^T| \odot |B^T||X||B|\right)|A|\left(2n_h+4n+7\right)\varepsilon+\O(\varepsilon^2)$$
	and the normwise boundary is equal to
	$$\|\hat{S}-S\|_1 \leq \|A^T\|_1\thinspace\|G\|_F\thinspace\|H\|_F\thinspace\|G^T\|_F\thinspace\|B^T\|_F\thinspace\|X\|_F\thinspace\|B\|_F\thinspace\|A\|_1\left(2n_h+4n+7\right)\varepsilon+O(\varepsilon^2)$$
\end{corollary}

\subsection{Components of the Toom-Cook error}
\label{sec:components}

The Toom-Cook error in Theorem \ref{theorem:TC-2D-error} states that
the bound is proportional to the product of three main terms: (a) the
product of the norms of the three convolution matrices $G$, $B^T$ and
$A^T$; (b) the product of the norms of the input $x$ and kernel $h$;
and (c) the sum of the errors from the linear transforms $\alpha^{(n)}$,
$\beta^{(n)}$ and $\gamma^{(n_h)}$.

The input $x$ and kernel $h$ can take on any value at execution time,
so their norms can be arbitrarily large if the input and kernel have
pathological values. Thus, the worst-case error arising from the
product of these norms can be arbitrarily large. However, most
inputs and kernels are unlikely to have pathological values. 
Furthermore, it is often more informative to study the relative error/stability of an algorithm,
i.e., the size of the error produced by the algorithm relative to the size of its inputs.  Interpreting the
error bounds derived in Theorems \ref{theorem:TC-1D-error} and \ref{theorem:TC-2D-error}, we see that
the relative errors are controlled by norms of the Vandermonde matrices and the summation order.
Furthermore, the errors arising from the linear transforms are polynomial, as shown
in Equation~\ref{eq:dot_prod}.  However, it should be noted that the relative condition number may depend on $x$ and $h$; see Appendix \ref{section.AppNormCond}.

The three matrices $G$, $A^T$ and $B^T$ are more problematic. As we
describe in more detail in Section \ref{sec:Toom-Cook}, $G$ and $A$
are (theoretically square although normally presented as rectangular)
Vandermonde matrices and $B^T$ is the inverse of (the square version
of) $A^T$.  The product of the norms of a square Vandermonde matrix
and its inverse grows at least exponentially with $n$
\cite{Pan16}. Thus, our bound on the error grows at least
exponentially with $n$.

The third component of the Toom-Cook algorithm error depends on
the values of $\alpha^{(n)}$, $\beta^{(n)}$, $\gamma^{(n_h)}$, which means that
it depends on the method of evaluation of the matrix-vector multiplication.

\subsection{Multiple Channels}

Note that DNN convolution is also normally computed across multiple
input channels. Both their input and kernel have the same number of
channels, and separate 1D or 2D convolutions are computed for each
channel. The resulting vectors or matrices (for 1D or 2D convolution
respectively) are summed pointwise to yield a single-channel
result vector or matrix. The separate convolutions for each channel
can be computed using Toom-Cook or indeed any convolution algorithm.

Toom-Cook convolution consists of three stages: pre-processing,
pairwise multiplication (Hadamard product), and post-processing. Lavin and Gray's DNN
convolution algorithm dramatically reduces the work of post-processing
for multi-channel convolution. The post-processing step is a linear
transform, so the sum of the transformed Hadamard products is equal to
the transform of the sum of the Hadamard products. Thus the post-processing
transform is applied just once after summing the Hadamard products,
rather than separately for each input channel \emph{before}
summation.

If we compute Toom-Cook convolution over $C$ channels we add the
results of Hadamard products $\sum_{c=1}^C(Gh_c \odot B^Tx_c)$ for
one-dimensional convolution and $\sum_{c=1}^C(GH_cG^T \odot B^TX_cB)$
for two-dimensional convolution, using the same matrices $G$ and $B^T$
on every channel. Thus we have the error less than or equal to:
%\begin{equation}
\vspace{-0.3cm}

\begin{equation}
\label{eq:channel_error_1D}
\begin{split} 
\|\hat{s}-s\|_1 %&\leq \\
&\leq \|A^T\|_1\thinspace C\thinspace\|G\|_F\thinspace max_c\thinspace\|h_c\|_2\thinspace\|B^T|_F\thinspace max_c\thinspace\|x_c\|_2R\varepsilon+O(\varepsilon^2) \nonumber
\end{split}
\end{equation}
Where $R=\alpha^{(n)}+\beta^{(n)}+\gamma^{(n_h)}+1+\lambda^{(C)}$,
%\end{equation}
%\Delta^{\alpha+\beta+\gamma+c}$$
$h_c$ and $x_c$ are the kernel and input vectors on channel $c$,
$\alpha^{(n)}$, $\beta^{(n)}$, $\gamma^{(n_h)}$ represent the dot product
errors and $\lambda^{(C)}$ is the error in pointwise summation.

\noindent For two dimensions we have:
\begin{equation}
\label{eq:channel_error_2D}
\begin{split}
\|\hat{S}-S\| %&\leq \\ 
&\leq \left(\|A^T\|_1\thinspace C\thinspace\|G\|_F \thinspace max_c\thinspace\|h_c\|_2\thinspace\|G^T\|_F\thinspace\|B^T\|_F \thinspace max_c\thinspace\|x_c\|_2\thinspace\|B\|_F\thinspace\|A\|_1\right)\thinspace R \thinspace \varepsilon + O(\varepsilon^2) \nonumber
\end{split}
\end{equation}
Where $R = 2\alpha^{(n)}+2\beta^{(n)}+2\gamma^{(n_h)}+1+\lambda^{(C)}$

\section{Modified Toom-Cook Algorithm}
\label{sec:modified}
A common method to reduce the number of terms in the linear transforms of Toom-Cook convolution is to use the so-called \textit{modified} version of the algorithm \footnote{See tensorflow source code at https://github.com/tensorflow/tensorflow/blob/9590c4c32dd4346ea5c356733\linebreak 36f5912c6072bf2/tensorflow/core/kernels/winograd\_transform.h\#L179-L186  \linebreak and MKL-DNN https://github.com/intel/mkl-dnn/blob/fa5f6313d6b65e8f6444c6900432fb07ef5661e5/doc/\linebreak winograd\_convolution.md}. In this section, we show that as well as reducing the number of FP operations in the linear transforms, the modified algorithm also significantly reduces the FP error in Toom-Cook convolution.

The main idea of the modified algorithm is to solve one size smaller
problem which means we use a kernel of the same size $n_h$ but an input of size $n-1$ instead of $n$.
Having computed such a convolution, we then \emph{modify} the output values in which the $n$th element of the input is included.

To minimize the number of operations in the Toom-Cook algorithm we construct
the polynomial $M(a)=\prod_i m_i(a)$ such that $deg(M(a)) = deg(s(a))+1$. We can further reduce the number of operation by using $M^{'}(a)$ where $deg(M'(a))=deg(M(a))-1=deg(s(a))$. \\
Then when we apply CRT instead of polynomial $s(a)=h(a)x(a)$ we obtain the polynomial
\begin{equation}
\label{eq:mod_win_res}
s'(a) = s(a) \; mod \; M'(a)
\end{equation}
Because all $m_i$ we use are monic, $M'(a)$ is also monic. We have $s(a) =
s'(a) + R \; M'(x)$, where the scalar $R$ is the coefficient of the variable
with highest degree in \linebreak $s(a)=h(a)x(a)$ i.e. $R = h_{n_h}x_n$. Finally we have:
\begin{equation}
\label{eq:mod_win}
s(a) = s'(a)+h_{n_h}x_nM'(a)
\end{equation}
where $s'(a)$ is a solution for convolution with one fewer inputs. \\
With this approach we need only $n-1$ points to construct polynomial $M'(a)$
instead of $n$ used in Section \ref{sec:Toom-Cook}. Formally, we use $M(a)=M'(a)(a- \infty)$ \cite{Blahut10}.  \\

Let us denote the matrices constructed by Toom-Cook algorithm for input $n$ as $G^{(n)}$, $B^{(n)T}$ and $A^{(n)T}$ and for modified Toom-Cook algorithm with input $n$ as $G^{m(n)}$, $B^{m(n)T}$ and $A^{m(n)T}.$

The modified Toom-Cook algorithm for input $n$ proceeds as follows:
\begin{itemize}
	\item Construct matrices $A^{(n-1)T}$, $G^{(n-1)}$ and $B^{(n-1)T}$ as for Toom-Cook for the problem of size $n-1$ with polynomial $M'(a)$.
	\item Construct matrix $G^{m(n)}$ by adding the $n$th row to the matrix
	$G^{(n-1)}$. This row includes zeros and a $1$ at the last position. Then $G^{m(n)}h = G^{(n-1)}h+h_{n_h}$.
	\item Construct matrix $A^{m(n)}$ in the same way by adding the $n$th row to
	the matrix $A^{(n-1)}$. This row includes zeros and a $1$ at the last position. Then  $A^{m(n)}x = A^{(n-1)}x + x_n$.
	\item Construct matrix $B^{m(n)}$ by adding the $n$th row and $n$th column to
	the matrix $B^{(n-1)}$. The last row includes zeros and a $1$ at the last
	position. The last column includes consecutive coefficients of polynomial $M'(a)$. Then $B^{m(n)}(G^{m(n)}h \odot A^{m(n)}x) =
	B^{(n-1)}(G^{(n-1)}h \odot A^{(n-1)}x)+h_{n_h}x_nM'(a)$
	\item Apply the Matrix Exchange theorem
\end{itemize}

The general form of the matrices obtained by modified Toom-Cook algorithm is as follows:

$$
{\scriptstyle
	G^{m(n)} = \left[
	\begin{matrix}
	\bf{G^{(n-1)}} \\
	\begin{matrix}
	0 & ... & 0 & 1
	\end{matrix}
	\end{matrix}
	\right]
	\qquad
	A^{m(n)T} =\left[
	\begin{matrix}
	\bf{A^{(n-1)T}} &
	\begin{matrix}
	0 \\
	... \\
	0 \\
	1
	\end{matrix}
	\end{matrix}
	\right]
	\qquad
	B^{m(n)T}=\left[
	\begin{matrix}
	\bf{B^{(n-1)T}} & \bf{\large 0} \\
	\begin{matrix}
	M_1(a) & ... & M_{n}
	\end{matrix}
	& 1
	\end{matrix}
	\right]}
$$

\subsection{Modified Toom-Cook Error Analysis}

Our Theorems \ref{theorem:TC-1D-error} and \ref{theorem:TC-2D-error} about error estimation apply both to Toom-Cook and modified Toom-Cook algorithms. However, we can distinguish error bounds for Toom-Cook and modified Toom-Cook. In this section, we present a FP error analysis for the modified version of Toom-Cook and show that it gives us tighter error bounds than Toom-Cook. As before, our error analysis is novel, but we rely on prior methods and results from Higham \cite{Higham02}, Demmel \cite{Demmel07}, Pan \cite{Pan16} and the work of Bini and Lotti \cite{Bini80} on the error for fast matrix multiplication. The presented bounds allow us to see the exact difference in FP error for both algorithms.

For modified Toom-Cook algorithm we have some zero elements in matrices that are independent of the parameters (points) we choose. The guaranteed properties of the modified Toom-Cook algorithm is that we have $n_h-1$ zeros elements and a single $1$ in the last row of $G$ matrix, $n-1$ zeros elements and a single $1$ in the last column of $B^T$ matrix and $n_o-1$ zeros elements and $1$ in the last column of $A^T$ matrix. In addition, we can observe that Toom-Cook matrices for input $n-1$ are submatrices for modified Toom-Cook for input $n$.

Let us denote the real result vector of modified Toom-Cook algorithm for input $n$ as $s^{m(n)}$ and convolution vector computing by modified Toom-Cook algorithm for input $n$ by $\hat{s}^{m(n)}$, similarly denote the real result vector of Toom-Cook algorithm for input $n$ by $s^{(n)}$ and computed result by ${\hat{s}}^{(n)}$. We put the vector $x^{(n-1)}$ as the vector of size $n-1$ where $x_i^{(n-1)}=x_i^{(n)}$ for $i=1...n-1$

\begin{theorem}
	The componentwise error for one-dimensional modified Toom-Cook for $q$th element of output is bounded by:
	\begin{equation}
	\label{eq:ModWinograd-1D-element}
	\begin{split}
	&|\hat{s}^{m(n)}_q-s^{m(n)}_q| = \\
	& |{A^{(n-1)T}}_{q:}|\left(|G^{(n-1)}|
	|h|\odot|B^{(n-1)T}||x^{(n-1)}|\right)\left(\gamma^{(n_h)}+\beta^{(n-1)}+\alpha^{(n-1)}+1\right)\varepsilon + O(\varepsilon^2)  \\
	&\mathrm{for} \; q=1,..,n_o-1 
	\end{split}
	\end{equation}
	\begin{equation}
	\begin{split}  
	|{\hat{s}_{n_o}}^{(m(n))}-s^{m(n)}_{n_o}| \leq& \\ |{A^T_{q:}}^{(n-1)}|\left(|G^{(n-1)}\right.&\left.||h|\odot|B^{(n-1)T}||x^{(n-1)}|+|h_{n_h}||{B_{n:}}^{m(n)T}||x|\right) \\
	&\left(max\left\{\left(\gamma^{(n_h)}+\beta^{(n-1)}+\alpha^{(n-1)}+1\right), \left(\beta^{(n)}+1\right)\right\}+1\right)\varepsilon + O(\varepsilon^2)\\
	\mathrm{for} \; q=n_o\;\;\;\;\;\;\;\;\;\;\;\;\; &
	\end{split}
	\end{equation}
\end{theorem}

\noindent The proof of this theorem is presented in Appendix \ref{sec:modWinograd_error},
along with corollaries for the case where $n_h \geq 3$.

{\fontsize{8}{10}
	\begin{algorithm}[h]
		\begin{multicols}{2}
			\SetAlgoNoLine
			\KwIn{$n_o$ - size of output, $n_h$ - size of kernel, \; $\{p_1,\cdots,p_{n-1}\}$ set of $n-1$ different points. \\
				$A^{(n-1)T}$, $B^{(n-1)T}$, $G^{(n-1)}$ - matrices constructed by Toom-Cook algorithm for $n-1$ points}
			\KwOut{Three matrices $A^{m(n)T}$, $G^{m(n)}$ $B^{m(n)T}$ for modified Toom-Cook convolution}
			$n=n_o+n_h-1$ \;
			\For{$i=1$ to $n-1$}{
				$N_i=\frac{1}{\prod_{j \neq i}(p_i-p_j)}$ \;
				$M_i =$ coefficient of the polynomial $\prod_k(a-p_k)$ of $i$th term
			}
			%macierz A
			$A^{m(n)T}_{1:n_o,1:n-1}=A^{(n-1)T}$ \;
			\For{$i=1$ to $n_o-1$}{
				$A^{m(n)T}_{i,n}=0$ \label{line:zero1}
			}
			$A^{m(n)T}_{n_o,n}=1$ \;
			$G^{m(n)}_{1:n_h,1:n-1}=G^{(n-1)}$ \;
			\For{$j=1$ to $n_h-1$} {
				$G^{m(n)}_{n,j}=0$ \label{line:zero2}
			}
			$G^{m(n)}_{n,n_h}=1$ \;
			$B^{m(n)T}_{1:n-1,1:n-1}=B^{(n-1)T}$ \;
			\For{$i=1$ to $n-1$}{
				$B^{m(n)T}_{i,n}=0$ \label{line:zero3}
			}
			\For{$j=1$ to $n$}{
				$B^{m(n)T}_{n,j}=M_{j}$
			}
		\end{multicols}
		\caption{Modified Toom-Cook algorithm}
		\label{alg:mod_Toom-Cook}
	\end{algorithm}
}

\subsection{Toom-Cook versus Modified Toom-Cook}

Comparing the componentwise error of Toom-Cook (\ref{eq:TC-1D-element}) and
modified Toom-Cook (\ref{eq:ModWinograd-1D-element}) algorithms, we observe
that the error of modified Toom-Cook is smaller. We can see from the formula of
modified Toom-Cook (\ref{eq:ModWinograd-1D-element}) that, in contrast to
unmodified Toom-Cook (\ref{eq:TC-1D-element}) the errors do not spread
uniformly over all output points. The idea of computing one size smaller
convolution and using the pseudo-point $\infty$ results in a different error
boundary for the last output points. Thus our comparision is split in two
parts: the error comparison for first $n_o-1$ output points and the error comparison
for the last output point.

Looking to the error formulas (\ref{eq:TC-1D-element} and
\ref{eq:ModWinograd-1D-element}) for the first $n_o-1$ output points we observe
that the submatrices used in error estimation of the modified Toom-Cook algorithm
with input of size $n$ are the same as in the Toom-Cook algorithm with input of
size $n-1$. This results from the modified Toom-Cook algorithm definition in
Section~\ref{sec:modified}. Thus we have the same error in modified Toom-Cook algorithm
with input $n$ as for Toom-Cook with input $n-1$. Since the ill-conditioning of
Vandermonde matrices increases exponentially with size, the error due to the
conditioning of matrices in modified Toom-Cook algorithm is significantly
smaller, although still exponential.

The second factor in the formulas for the first $n_o-1$ output points of both
algorithms is the error from floating point operations. The error due to the
dot product has tighter boundaries for modified Toom-Cook
($\gamma^{(n_h)}+\beta^{(n-1)}+\alpha^{(n-1)}$) than for Toom-Cook \linebreak
($\gamma^{(n_h)}+\beta^{(n)}+\alpha^{(n)}$), if we assume the same method of
summation in both algorithms. It is clear that the worst-case error of
summation of $n-1$ elements is not larger than the worst-case error of
summation of $n$ elements, so $\beta^{(n-1)} < \beta^{(n)}$ and $\alpha^{(n-1)}
< \alpha^{(n)}$.

Because both components of the error in the first $n_o-1$ output points are
smaller in modified Toom-Cook, we can safely conclude that the overall worse
case error in these points is smaller than in the unmodified Toom-Cook
algorithm.

\subsubsection{Error in Modified Points}

To compare the error of the last output points for both algorithms, we observe
from the definition of the modified Toom-Cook algorithm that the error from
matrix elements is bounded by a sum of the error for Toom-Cook at size $n-1$
and the error provided by the last row of matrix $B^{m(n)T}$. The values in the
last row of the matrix are exactly the same as for a row constructed for an
interpolation point $p_i = 0$, but in \emph{reverse} order. Thus the overall
error from matrix elements for modified Toom-Cook is not larger than for
Toom-Cook.

The error from the method of dot product computation for the last output point
of the \emph{modified} algorithm is equal to
$\gamma^{(n_h)}+\beta^{(n-1)}+\alpha^{(n-1)}+2$ or $\beta^{(n)}+1$. In both
cases this value is smaller than the corresponding value
$\gamma^{(n_h)}+\beta^{(n)}+\alpha^{(n)}+1$ in the unmodified Toom-Cook error
estimation.

Although the error for \emph{both} Toom-Cook and modified Toom-Cook algorithms
grows exponentially, the error for all output points for modified version is
smaller than for the original unmodified Toom-Cook algorithm.

\section{Empirical Measurement of FP Error}
\label{sec:empirical}
The formal error analysis that has appeared in earlier sections of the
paper is a worst-case analysis. However, even if the worst-case error is potentially very
large, it is important to know something about the typical error that arises
in practice. Almost all formal analyses of FP error are worst-case
analyses. For example, all the analyses in Higham's standard textbook
on FP error are worst-case estimates \cite[p. 48]{Higham02}. Studies of average
case probabilistic FP error are possible in principle, but they rely
on assumptions about the distribution of errors that are difficult to
verify.  For example, Kahan, who won the Turing award for his
contributions to FP numerical analysis, has argued that FP rounding
errors are typically not random, often correlated, and often behave
more like discrete variables than continuous ones \cite{Kahan98}, which
makes average case analyses unreliable.

The focus of our work is on understanding and reducing the FP error in
fast DNN convolution. So rather than deal with the many pitfalls of
formal average case analysis, we make empirical measurements
of the FP errors.

To measure the error in Toom-Cook convolution, we first need the
algorithm for a specific size, which is defined by $n_h$, $n_o$ and
the $n = n_h+n_o-1$ real-valued \textit{points} that are used to
sample the polynomials corresponding to the input and kernel. We study
over $40000$ of point selections and find that the the values
of these \textit{points} has a huge impact on the FP error (see Section
\ref{sec:selecting-points}).

When generating the $A^T$, $G$ and $B^T$ matrices using these points,
we represent all values symbolically rather than as FP numbers.  This
allows us to generate exact values in each element of the convolution
matrices. Once the elements have been generated, we then convert each
value to the nearest representable FP number. Recall that  $A^T$, $G$
and $B^T$ are constant matrices, so we compute them as accurately as
possible ahead of time.

FP numbers are constructed as a logarithmic sampling of the real number
line, and the range $(-1,1)$ is where they have most precision. The
values of trained DNN weights are overwhelmingly concentrated in this
range in practice. Since we are interested in differentiating the inherent error in the
convolution algorithms, not just in the context of specific networks,
we would like to know something about the average case error
\emph{independent} of any specific dataset or network. For this reason,
rather than model inputs and kernels with specific distributions drawn
from real networks, we model them as random variables with uniform
distributions in the range$(-1,1)$.

We compute the error as the L1 norm $||\cdot||_1$ between the result of
the convolution, and an approximation of the numerically correct
result. We find our approximation of the numerically correct result
using \textit{direct} convolution in 64-bit double precision FP. We
compute the error as the L1 norm $||\cdot||_{L1}$ from the difference
between the result computed using the proposed method and our
\textit{approximately correct} result.

$$\|A^T(Gh \odot B^Tx) - h \ast x\|_{L1}$$
$$\|A^T(GHG^T \odot B^TXB)A - H \ast X\|_{L1}$$
%=\sum_{i=0}^{n_o}|solution(s_i)-solution\_{direct}(s_i)|$$

For $1$ and $2$ dimensions respectively. Where for two vectors: $a = \begin{bmatrix}a_1 & \cdots &a_n\end{bmatrix}$ and \linebreak $b = \begin{bmatrix}b_1& \cdots &b_n\end{bmatrix}$ the norm $\| \cdot \|_{L1}$ is equal to sum of absolute from a difference between corresponded elements: $\|a - b\|_{L1}=\sum_i|a_i - b_i|$. For two matrices $A$ and $B$ the formula is $\|A-B\|_{L1}=\sum_{i,j}|A_{i,j}-B_{i,j}|$.

We found that $5000$ iterations of random testing was sufficient for
the average error to become stable. In all experiments we use a kernel
of size 3 for 1D and $3 \times 3$ for 2D convolution, which are the
most common sizes in real DNNs.

We empirically compared the numerical error of convolution algorithms generated by the Toom-Cook and modified Toom-Cook methods. The error is extremely sensitive to the \textit{points} that are selected. However, for a given set of \textit{points}, replacing one of them with the $\infty$ pseudo-point almost always reduces the error. For sets of points that otherwise result in a low error, we observed that modified Toom-Cook gave a reduction in numerical error from $20\%$ for kernel size 3 and output size $16$, to over $70\%$ for kernel size 3 and output $2$. Throughout the remainder of the paper, we use the $\infty$ pseduo-point to indicate where the modified algorithm is used.

\section{Selecting Points and Orders of Evaluation}
\label{sec:selecting-points}

The Toom-Cook method gives the mathematically correct result using any
sufficiently large set of distinct sampling points. However, there is
a large difference between the FP error using different sets of
points and there is no known systematic method for selecting the
best points to minimize the error \cite{Vincent17}. The points we use have an impact on the norm of matrices $G$, $A^T$ and $B^T$ as well as for the values of $\alpha^{(n)}$, $\beta^{(n)}$ and $\gamma^{(n_h)}$ in error formula in theorems (\ref{theorem:TC-1D-error}) and (\ref{theorem:TC-2D-error}).

In
this section we study the problem of selecting points
experimentally. In the first stage we simply evaluated the random sets
of points, and quickly discovered that (1) some sets of points are
much better than others; and (2) not just the value of the points, but
their ordering is important. The same set of points considered in a
different order give quite different numerical errors.

\subsection{Canonical Summation Order}
Different orderings of the same points give different
answers because of the order of evaluation. Different point
orderings result in different orderings of the values within the
$A^T$, $G$, $B^T$ matrices. The transform steps of Toom-Cook
convolution involve multiplying each of the input, kernel, and output
by one of these matrices. If we change the order of entries in the
matrix, then we change the order of evaluation at execution time,
which causes different FP rounding errors. Some point orderings were
better than others, but it was difficulty to predict the good ones
ahead of time.

Rather than searching different orderings of points, we propose to fix
the order of evaluation, so that all orderings of the same set of
points will be evaluated in the same order. The remaining problem is
to pick a canonical order of evaluation that works well in
practice. Each row of the $A^T$, $G$, $B^T$ matrices is used to
compute single dot product within a linear transform, and we specify a
canonical ordering for evaluating each of these dot products.

\begin{SCfigure}[][h]

	\subfloat[Linear]{
		\includegraphics[scale=0.23]{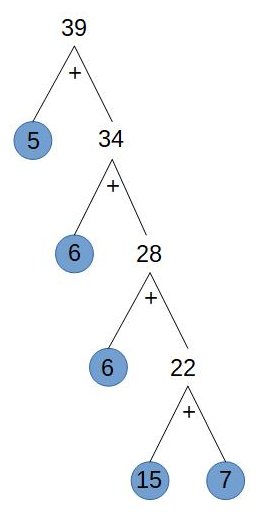}
		\label{linear}
	}
	\subfloat[Huffman]{
		\includegraphics[scale=0.25]{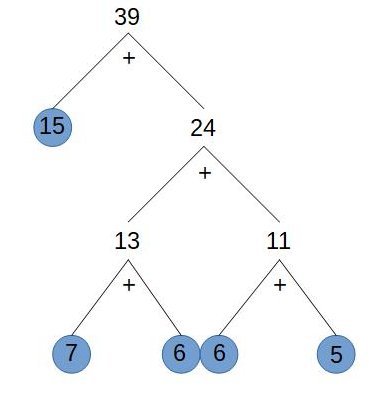}
		\label{huffman}
	}
	\caption{Linear and Huffman tree (canonical) summation methods. Our canonical ordering has two main advantages. It reduces the FP
		summation error in the linear transforms by improving the order of
		evaluation. It also ensures that we get the same error if we evaluate
		the same set of \textit{points} in different orderings; the order of
		evaluation is determined by the Huffman tree, not by the order in
		which the points are presented.
	}
	\label{fig:linear-huffman}
	
\end{SCfigure}

We build a Huffman \cite{Huffman52} tree using the absolute values of
each row, that is used to specify the order of summation. We also use
simple heuristics to break ties between coeffients with the same
absolute values.
A basic principle of accurate FP summation is to try to sum smaller
values first, as shown in Figure \ref{fig:linear-huffman}. Our Huffman
tree is based purely on the values of the rows of our constant
matrices; we build the tree at the algorithm design time, when the
input and kernel are unknown. This makes it much easier for us to
search empirically for good sets of \textit{points} because we need
only consider their value, not their ordering. In addition this method
allows us to use different order of summation for every row of
matrices that is not possible to obtain by points permutation.

There are a great deal of different summation methods. They were investigated in detail by Rump (\cite{Rump08I}, \cite{Rump08II}) and Demmel (\cite{Demmel04}). They guarantee the accurate or nearly accurate result of dot product computations. However they require additional arithmetic operations either (1) for a compensated summation, or (2) to sort elements before summation, that slow down the convolution computations. These methods trade-off increased accuracy against increased computation cost, which is similar to the mixed-precision method we propose in Section \ref{sec:mixed}.

Our canonical evaluation order is not guaranteed to sum in increasing
order of absolute value, because the execution time inputs might
contain large or small values. But in practice our canonical ordering
does much better than arbitrary orderings. We tested our approach with
the setup describe in Section \ref{sec:empirical}. Across a range of
convolution sizes using various points we found roughly a $14\%$
improvement in accuracy for $1D$ and $12\%$ for $2D$ compared with the
same selection of points in an arbitrary order.  All subsequent test results
presented in this paper use our Huffman summation for the transforms.

\subsection{Point Choice}

We empirically evaluated over $40000$ of random selections of
values for the \textit{points} that are used to construct the $G$,
$A^T$ and $B^T$ matrices that are used to perform the linear
transforms. We quickly found that it is very easy to find sets of
points that cause huge FP errors, and rather more difficult to find
better points.

There is no single recognized method for selecting \textit{points}
that minimize the FP error. However, the Chebyshev nodes are known to
improve the conditioning of polynomial interpolation \cite{Higham02},
which is an important step of Toom-Cook convolution. Results for the
FP error of using the Chebyshev nodes can be found in Appendix
\ref{sec:Chebyshev}. In general, the Chebyshev nodes are orders
of magnitude better than typical random point selections.

There is, however, some common wisdom in the literature on another
approach to selecting points to reduce the computation in the linear
transforms. In general, the points $\{0, -1, 1, \infty\}$ are good for
reducing these costs, assuming that the code to implement convolution
exploits these values. Multiplication by 1 or -1 can simply be
skipped, and multiplication by zero allows both the scaling and
addition to be skipped. Fortunately, eliminating FP operations also
eliminates their associated error, so these points are also suitable
to reducing FP error.

Problems start to arise where we need more than just these four basic
points. In general, researchers agree that selecting small simple
integers and fractions are good choices for reducing the required
number of scalings and additions. We also found this type of points to
be good for reducing the FP error. But there is no agreed-upon method in
the literature for selecting between different values such as $2, -2,
\frac{1}{2}, -\frac{1}{2}, \frac{3}{2}, \frac{2}{3}$, etc.

To help us find good sets of these simple values to reduce the FP
error, we developed the following rules which act as a heuristic to
guide our search. The size of the kernel and output block determine
the number of points needed. We start with the basic points $\{0, -1,
1, \infty\}$, which work well when four points are needed.
We perform our search for sets of good points for output $n$ based on
the good sets of points for output $n-1$. We establish a set of
potentialy interesting points according to below rules as rationals
with numerator in $\{-4, -3,-2,-1,0,1,2,3,4\}$ and denominator in
$\{1,2,3,4\}$. This gives us a set $P$ of $23$ possible points.
% Since the number of all subsets size $n$ from $23$ elements set is much too big to check all possibilities, we perform the computations only for selected ones.

\subsubsection*{Our Algorithm}
\begin{itemize}
	\item We start with a set  $P_{n-1}$ of $n-1$ good points $p_1,...,p_{n-1}$
	\item We construct new sets of points by adding $p_n$. \\
	$P_{n-1}\cup \{p_n\}$  $\forall p_n$ such that $p_n \in P$ and $p_n \neq p_i$ $\forall i = 0,..,n-1.$
	\item If $n$ is even we construct a new set by dropping $p_j$ and adding two new points $p_k$ and $-\frac{1}{p_k}$. \\
	As $n$ is even we have at least one point $p_j$ without symmetry, that is,
	$p_j \in P_{n-1}$ and $-\frac{1}{p_j} \notin P_{n-1}$. We drop the point $p_j$ and add instead all pairs $p_k$ and $-\frac{1}{p_k}$\\
	$P_{n-1}\setminus \{p_j\} \cup \{p_k,-\frac{1}{p_k}\}$ $\forall p_k \in P$, $p_k, -\frac{1}{p_k} \neq p_i$ $\forall i=0,..,n-1$.
	\item If there are different sets of points for $1D$ and $2D$ we check both sets for $1D$ and $2D$ in parallel.
\end{itemize}

\noindent The resulting sets of ``good'' points are presented in Table
\ref{tab:alternative-pointsfl}. The basic four points are always $\{0,
-1, 1, \infty\}$.  Table \ref{tab:alternative-pointsfl} shows that when
we add a fifth point, we found empirically that $\frac{1}{2}$ is the
best point to add for both 1D and 2D convolution. Occasionally, when
moving to the next larger number of points we remove an existing point
and add two new ones, such as when we add the eight point for 1D
convolution. Note that the FP error versus direct convolution grows
rapidly with the number of points. The increased error is due to the
algorithm becoming less accurate with more \textit{points}. The growth
in error appears to be roughly exponential in the number of points in
practice. However, the growth in error is not smooth (see
Figure~\ref{fig:error_increasing}).

\definecolor{Gray}{gray}{0.9}
\begin{table}[h]
	\caption{Example points for Toom-Cook in $FP32$ with kernels of size 3 (for 1D) or $3 \times 3$ (for 2D). We start with set of $4$ points $P_4=\{0,-1,1,\infty\}$. Then we present the set of points for different cardinality as a additions and/or subtraction of sets of points.}
	\label{tab:alternative-pointsfl}
	\begin{minipage}{\columnwidth}
		\begin{small}
			\begin{center}
				\begin{tabular}{c|lcc|lcc}
					$n$    & Points 1D & Error 1D & $n_o$ & Points 2D & Error 2D & $n_o$ \\
					\toprule
					0  & Direct convolution  & 1.75E-08 & $1$  & Direct conv. &  4.63E-08 &  $1 \times 1$ \\
					\rowcolor{Gray}4  & $P_{4} = \{0,-1,1, \infty\}$  & 2.45E-08 & $2$ & $P_{4}$   &  7.65E-08 &   $2 \times 2$   \\
					5  & $P_{4} \cup \{\frac{1}{2}\}$      & 5.19E-08  & $3$ & $P_{4} \cup \{\frac{1}{2}\}$ &  2.35E-07 &  $3 \times 3$ \\
					6  & $P_{4} \cup \{\frac{1}{2}, -3\}$   & 6.92E-08  & $4$ & $P_{4} \cup \{\frac{1}{2}, -2\}$         &   3.29E-07 &  $4 \times 4$ \\
					7  & $P_{4} \cup \{\frac{1}{2}, -\frac{1}{2}, -3\}$ & 9.35E-08 & $5$ & $P_{4} \cup \{\frac{1}{2}, -2, -\frac{1}{2}\}$ & 6.81E-07 &  $5 \times 5$ \\
					\rowcolor{Gray}8  & $P_8 = P_{4} \cup \{\frac{1}{2}, -\frac{1}{2}, 2, -2\}$ & 1.15E-07 & $6$ & $P_{8}$ & 8.79E-07 &  $6 \times 6$ \\
					9  & $P_8 \cup \{-\frac{1}{4}\}$      & 2.34E-07  & $7$ & $P_8 \cup \{-\frac{1}{4}\}$  &  3.71E-06   &  $7 \times 7$\\
					10  & $P_{10}=P_{8} \cup \{-\frac{1}{4},4\}$                & 3.46E-07 & $8$ &  $P_{10}$  &  7.35E-06            &  $8 \times 8$\\
					11  & $P_{10} \cup \{\frac{1}{4}\}$      & 5.91E-07  & $9$ &  $P_{10} \setminus\{0\} \cup \{\frac{3}{4} , -\frac{4}{3}\}$ & 2.2E-05  & $9 \times 9$ \\
					\rowcolor{Gray}12  & $P_{10} \cup \{\frac{3}{4},-\frac{4}{3}\}$ & 7.51E-07 & $10$ &  $P_{10} \cup \{\frac{3}{4},-\frac{4}{3}\}$   & 3.22E-05  &   $10 \times 10$\\
					13  & $P_{10}\cup\{\frac{3}{4},-\frac{4}{3},\frac{1}{4}\}$      & 1.32E-06  & $11$ & $P_{10}\cup\{\frac{3}{4},-\frac{4}{3},\frac{1}{4}\}$   & 1.09E-04   &  $11 \times 11$ \\
					14  & $P_{14}= P_{10}  \cup\{\frac{1}{4},-\frac{3}{4},\frac{4}{3},-4\}$ & 1.84E-06 & $12$ & $P_{14}$ & 1.99E-04  &  $12 \times 12$ \\
					15  & $(P_{14}\setminus\{0\})$ $\cup \{\frac{2}{3},-\frac{3}{2}\}$ & 3.42E-06 & $13$ & $P_{14} \setminus\{0\}$ $\cup\{\frac{3}{4}, -\frac{4}{3}\}$& 5.54E-04 & $13 \times 13$ \\
					\rowcolor{Gray}16  & $P_{14} \cup \{\frac{2}{3},-\frac{3}{2}\}$        & 4.26E-06 & $14$ &  $P_{14}$ $\cup\{\frac{3}{4}, -\frac{4}{3}\}$     &  8.8E-04          &  $14 \times 14$ \\
					17  & $P_{14} \cup \{\frac{2}{3},-\frac{3}{2}, -\frac{2}{3}\}$   & 1.35E-05 & $15$ & $P_{14}$ $\cup\{\frac{2}{3}, -\frac{3}{2}, \frac{3}{2}\}$     &  1.07E-02           &  $15 \times 15$ \\
					18  & $P_{14} \cup \{\frac{2}{3},-\frac{3}{2}, -\frac{2}{3}, \frac{3}{2}\}$  & 2.24E-05 & $16$ &  $P_{14} \cup \{\frac{2}{3},-\frac{3}{2}, -\frac{2}{3}, \frac{3}{2}\}$   &   1.93E-02          &  $16 \times 16$
				\end{tabular}
			\end{center}
			\bigskip\centering
		\end{small}
	\end{minipage}
\end{table}%

\subsection{Error Growth}
If we consider the growth in 1D error from 7 to 8 points, the error
grows from $9.35\times 10^{-8}$ to $1.15\times 10^{-7}$, which is a
factor of around $1.12\times$. In contrast the growth in error from 8
to 9 points is $1.15\times 10^{-7}$ to $2.34\times 10^{-7}$, which is
a factor of $2.03\times$. This is not a coincidence. The empirically
good solution that we found when seven points are used for 1D
convolution is $\{0, -\frac{1}{2}, \frac{1}{2},-1, 1, -3,\infty\}$. In contrast when eight points are needed, the good
solution is $\{0, -\frac{1}{2}, \frac{1}{2},-1, 1, -2,2,\infty\}$. Among the eight points there is a symmetry between the four
values $\{-\frac{1}{2},\frac{1}{2}, -2, 2\}$, which are negations and
reciprocals of one another. As we discuss in Section
\ref{sec:discussion}, these symmetries reduce FP error.

In contrast, where 7 points are needed, the points $\{-\frac{1}{2},\frac{1}{2}, -3\}$ do not cancel in the same way, and so the error for
seven points is larger than a smooth growth in error with points would
suggest. Note that the appearance of the point $-3$ as the sixth
selected point for 1D convolution was a great surprise to us.
However, $-3$ has just two significant binary digits, so there is no
representation error of $-3$ in FP, and multiplication by $-3$ causes
a very small error. Further, when computing differences between pairs
of points for the $B^T$ matrix, $-3 -(-1) = -2$ and $-3-1 = -4$, both
of which are even powers of two.

For 2D convolution, small differences mean that $-3$ is very slightly
worse than $\frac{1}{2}, 2$ or $-2$ and is not selected.

Given these trends it is reasonable to ask where are the good
trade-offs between computation and error growth. The number of
\textit{general multiplications} is simply the number of
\textit{points} in the convolution algorithm. Using good point
selections has no additional cost over bad ones, but greatly reduces
the error. Similarly, our Huffman summation reduces the error at no
additional computation cost. There is no single best trade-off because
it depends on the required accuracy. However, using our point
selections and other methods it may be possible to increase the
output block size by 1-3 units.

\begin{figure}
	\begin{center}
		\subfloat[1D]{
			\includegraphics[scale=0.5]{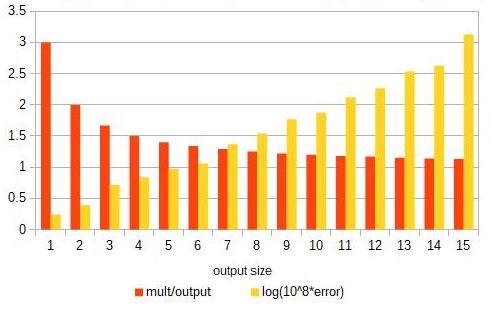}
			\label{fig:1D-blocks}
		}
		\subfloat[2D]{
			\includegraphics[scale=0.29]{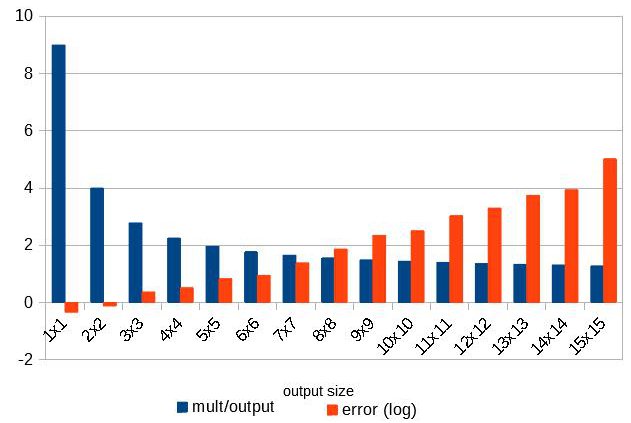}
			\label{fig:2D-blocks}
		}
		\caption{Number of multiplication and error for single output point for different input block sizes in one- and two-dimensional Toom-Cook convolution}
		\label{fig:Blocks}
	\end{center}
\end{figure}

\begin{SCfigure}[][h]
	\includegraphics[scale=0.4]{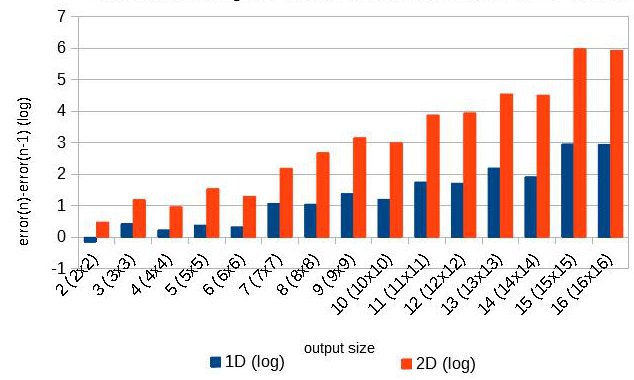}
	\caption{Increasing error in one- and one-dimensional Toom-Cook convolution. The vertical axis show the difference of logarithms from error of convolution computed with $n$ and $n-1$ points, to show error growth.}
	\label{fig:error_increasing}
\end{SCfigure}

When we examine the measured errors in Figure
\ref{fig:error_increasing} we see that even using good point
selections the measured error increases roughly exponentially with the
size of the convolution. The average measured error grows at a rate
that is compatible with the worst-case bound proven in
Section~\ref{sec:Toom-Cook}.

Our analysis in Section \ref{sec:2d-analysis} also suggests that the
FP error for 2D convolution grows quadratically more quickly than the
error for 1D convolution. This is borne out in Figure
\ref{fig:error_increasing}, where the 2D error as a function of the 1D
error is approximately $f(x) = \frac{x^{2}}{2.97}$.

As a further check on the consistency of our measurements, we also
implemented a \textit{running error} analysis for 1D Toom-Cook
convolution. Running error analysis
\cite[p. 65]{Higham02} is an empirical method that
computes a partial bound based on actual values alongside the
executing algorithm. In our experiments we found that the running
error closely matched the exponential rate of growth of the average
error, with the running error $4.63\times$ to $7.51\times$ times the
average error.

\subsection{Discussion of point selection}
\label{sec:discussion}

The point selection affect both components of forward error: conditioning of used matrices and floating point error. The goal of our tests is to find a good balance between them. Based on the our theoretical analysis, literature (\cite{Higham02}, \cite{Gautschi74/75}, \cite{Gautschi90}) and our empirical experiments those presented in Figure $4$ and many other tests that give us much larger FP error, it is possible to explain why some points are better than others.

One common way to mitigate the Vandermonde matrices ill-conditioning is use Chebyshev nodes. We tried this approach (see Appendix \ref{sec:Chebyshev}) but found that this did not perform well.
The size of Vandermonde matrices used in DNN convolution is relatively small. The error generated by ill-conditioning grows exponentially. But for the small convolutions in DNNs, the error from ill-conditioning is not so large as to outweigh the error from floating point operation and representation. Chebyshev nodes are mostly irrational, so they can not be represented exactly as FP values. This representation error propagates throughout the algorithm.

In addition, we are interested in the accuracy of discrete
convolution. There is a single correct answer (with known degree) to
the interpolation in Toom-Cook convolution. If we were computing
with infinite precision we could compute the correct polynomial
precisely. This is somewhat different to another common use of
interpolation, which is to \textit{estimate} a polynomial where the
degree is unknown. The advantages of Chebyshev interpolation
points to mitigate Runge's phenomenon does not help Toom-Cook
convolution. We can not ignore conditioning entirely. But our goal is to find a set of points that will minimize both factors: problem conditioning \textit{and} floating point error.

The four basic points $\{0,-1,1,\infty\}$ are almost always a good choices. In particular $0$ and $\infty$ result in guaranteed zeros in all three matrices $A^T$, $G$ and $B^T$ which cause no FP error. But we have to explain the point selection for bigger sets. We note that some clear rules for point selection emerge from our empirical study.

Firstly, we should use pairs of points that differ in sign (positive/negative), and pairs of reciprocal points --- see Table \ref{tab:alternative-pointsfl}. If we use point $p$ then using $-p$, $\frac{1}{p}$ and $-\frac{1}{p}$ allow us to get better accuracy than introducing other point. Positive/negative pairs of points generate lower elements in matrix $B^T$ \ref{alg:Toom-Cook}. Looking for the formula of matrix $B^T$ elements construction we can noticed that multiplication $(a-p)(a+p)=-p^2+0a+a^2$ results in zero coefficient of second term. Reciprocal points introduce the coefficient of the first term equal to $1$ in multiplication $(a-p)(a-frac{1}{p}p)=1-\frac{p^2+1}{p}a+a^1$ that not introduce any additional scalling error. The opposite points are also known to be good for Vandermonde matrices conditioning (\cite{Gautschi74/75}, \cite{Gautschi90}).

Secondly, the floating point error boundary depends directly on the values used in operations (see formula \ref{eq:dot_prod}). That means that we should look for the small values close to one to reduce the error. Putting it together with the previous observation we can say that choosing rational points minimizing numerator and denominator is a good strategy. This approach to points selection has also a positive impact for the floating point error. Assuming that each of kernel and input elements have a similar distribution using scaling factors with similar order of magnitude it is more likely to avoid cancellation error while computing dot product. The interesting point is that for bigger sets this is not always leading rule. We found that $\frac{1}{4}$ together with $(0,-1,1,-1/2,1/2,-2,2)$ works better than $\frac{1}{3}$ for $9$ points (see tables \ref{tab:alternative-pointsfl}, \ref{tab:alternative-pointsfld}). We explain this phenomena later in this section.

Thirdly, representation error of matrices elements have a big impact for the accuracy of the result. As we mentioned below the representation error propagates through all floating point operations and therefore can significantly grow up. This is why the exactly represented points work best in investigated algorithm (see tables \ref{tab:alternative-pointsfl}, \ref{tab:alternative-pointsfld}).

Finally, as we described in subsection (\ref{subsec:LinTransf}) the error from multiplication while computing the dot product affects on accuracy as well. The elements equal to power of $2$ do not introduce any error from scaling and therefore keep floating point error smaller. This and previous observation explain us why point $\frac{1}{4}4$ is better then $\frac{1}{3}$. The value $\frac{1}{4}$ in contrast to $\frac{1}{3}$ is exactly represented and do not introduce any error from multiplication.

Thus there is not a simple algorithm to choose a set of good points for Winograd algorithm. Our theoretical analysis allow us to identify all components of the error and dramatically narrow the search space. With that knowledge it is possible to check empiricaly which of the narrow sets of points works the best in practice.

\section{Mixed-precision pre-/post-processing}
\label{sec:mixed}
We often apply the same kernel to the set of many different inputs. Similarly we often compute convolution with different kernels for the same input data. Thus, the pre-/post processing of each input, kernel and output are done just once, whereas the transformed data is used many times. One way to improve accuracy is to use a mixed-precision algorithm, where the pre-/post processing is done in higher precision, while the inner loops that perform the pairwise multiplication (Hadamard product) are computed in standard precision. This approach lowers the value of machine epsilon $\varepsilon$ for the linear transforms in the error formulsa in Theorems (\ref{theorem:TC-1D-error}) and (\ref{theorem:TC-2D-error}).

Table \ref{tab:alternative-pointsfld} shows the point selection and
measured errors for a mixed-precision Toom-Cook that performs the
pre-/processing in FP64 and all other processing in FP32. We found
that the mixed precision algorithm reduced the error in both $1D$ and
$2D$ by up to around $40\%$ see column "Ratio" in table
(\ref{tab:alternative-pointsfld}). The result is that for the same
level of error, the mixed-precision algorithm can often allow an ouput
size that is one larger. We observe that in most cases the same sets
of points worked best for convolution computed in FP32 and in mixed
precision. Where there are differences, in most cases this is the
result of a slight difference in the order in which points are
selected when the number of points is odd. In the mixed-precision
rounding errors during the pre-/post processing steps become a little
less important because intermediate values are represented in FP64.

The cost of computing in double precision is significant. Modern
processors typically use vector arithmetic units, and the throughout
of double precision (FP64) is normally just half of single precision
(FP32). On graphics processing units (GPUs) the disparity between
single and double precision can be much larger. Furthermore, data
conversions between single and double precision are normally needed,
which further increase the computing cost.

A growing trend in deep neural networks is to store inputs and kernels
in FP16 precision in memory, but to compute in FP32. Most Intel and
ARM processors do not support FP16 arithmetic. But they provide fast
instructions for converting from FP16 values to FP32 to allow FP16
storage and FP32 computation. Many GPUs support native FP16
arithmetic, but it is not commonly used for DNN convolution. FP16
errors accumulate too rapidly for DNN convolution, particularly the
errors arising from summation across channels. Recent NVidia GPUs
provide co-called \textit{tensor cores}, which are specifically aimed
at storing data in FP16 and computing in FP32. The tensor cores accept
FP16 inputs, and compute a multiple-input \textit{fused multiply and
	add} operation, which produces a FP32 result. These operations could
be used to implement mixed-precision FP16/FP32 linear transforms for
Toom-Cook convolution.

\begin{table}[h]

	\caption{Example points for Toom-Cook method computations in $FP32$ with transforms in $FP64$ with kernels of size 3 (for 1D) or $3 \times 3$ (for 2D). We start with set of $4$ points $P_4=\{0,-1,1,\infty\}$. Then we present the set of points for different cardinallity as a additions and/or substractionf of sets of points. The column Ratio present the ratio between the error of convolution computations with transforms in $FP64$ and convolution computations in $FP32$.}
	\label{tab:alternative-pointsfld}
	\begin{minipage}{\columnwidth}
		\begin{footnotesize}
			\begin{center}
				\begin{tabular}{c|lccc|lcc}
					\toprule
					$n$    & Points 1D & Error 1D & $n_o$ & Ratio & Points 2D & Error 2D & Ratio \\
					\toprule
					0 & Direct convolution  & 1.75E-08  & 1 & 1 & Direct convolution &  4.63E-08  & 1 \\
					\rowcolor{Gray}4 & $P_{4}=\{0,-1,1, \infty\}$ & 1.87E-08 & 2 & 0.76 & $P_{4}$ & 5.27E-08 & 0.69 \\
					5 & $P_{4} \cup \{3\}$ & 3.66E-08  & 3 & 0.71 & $P_{4} \cup \{3\}$ & 1.62E-07 & 0.69 \\
					6 & $P_{4} \cup \{3,-\frac{1}{2}\}$ & 4.41E-08 & 4 & 0.64 & $P_{4} \cup \{3,-\frac{1}{2}\}$ & 2.14E-07 & 0.65 \\
					7 & $P_{4} \cup \{3,-\frac{1}{2}, \frac{1}{2}\}$ & 6.09E-08 & 5 & 0.65 & $P_{4} \cup \{3,-\frac{1}{2},\frac{1}{2}\}$ & 3.69E-07 & 0.54 \\
					\rowcolor{Gray}8  & $P_{8}=P_{4} \cup \{-\frac{1}{2},\frac{1}{2},-2,2\}$& 6.97E-08 & 6 & 0.61 & $P_{8}$ & 5.18E-07 & 0.59 \\
					9  & $P_{8} \cup \{-\frac{1}{4}\}$ & 1.55E-07 & 7 & 0.66 & $P_{8} \cup \{4\}$ & 2.42E-06 & 0.65\\
					10 & $P_{10}=P_{8} \cup \{-\frac{1}{4},4\}$ & 2.09E-07 & 8 & 0.6 & $P_{10}$ & 4.41E-06 & 0.6 \\
					11  & $P_{10} \cup \{\frac{1}{4}\}$ & 3.64E-07 & 9 & 0.62 & $P_{10} \setminus \{0\} \cup \{\frac{3}{4},-\frac{4}{3}\}$ & 1.27E-05 & 0.58 \\
					\rowcolor{Gray}12 & $P_{12}=P_{10} \cup \{\frac{3}{4},-\frac{4}{3}\}$ & 4.50E-07 & 10 & 0.6 & $P_{12}$ & 1.89E-05 & 0.59 \\
					13 & $P_{12} \cup \{\frac{1}{4}\}$ & 8.25E-07 & 11 & 0.63 & $P_{12} \cup \{-4\}$ & 6.38E-05 & 0.59 \\
					14 & $P_{14}=P_{12} \cup \{\frac{1}{4},-4\}$ & 1.11E-06 & 12 & 0.6 & $P_{14}$ & 1.14E-04 & 0.57 \\
					15 & $P_{12} \setminus \{0\} \cup\{\frac{2}{3},-\frac{3}{2}\}$ & 2.17E-06 & 13 & 0.63 & $P_{12} \setminus \{0\} \cup\{-\frac{3}{4}, \frac{4}{3}\}$ & 3.08E-04 & 0.56\\
					\rowcolor{Gray}16 & $P_{16}=P_{14} \cup \{\frac{2}{3},-\frac{3}{2}\}$ & 2.78E-06 & 14 & 0.65 & $P'_{16}=P_{14} \cup \{-\frac{3}{4},\frac{4}{3}\}$ & 4.95E-04 & 0.56 \\
					17 & $P_{16} \cup \{-\frac{2}{3}\}$ & 8.43E-06 & 15 & 0.62 & $P'_{16} \cup \{\frac{3}{2}\}$ & 5.93E-03 & 0.55 \\
					18 & $P_{18}=P_{16} \cup \{-\frac{2}{3},\frac{3}{2}\}$ & 1.39E-05 & 16 & 0.62 & $P_{18}$ &  1.04E-02 & 0.54
				\end{tabular}
			\end{center}
			\bigskip\centering
		\end{footnotesize}
	\end{minipage}
\end{table}%

\section{Multiple channels}
The proposed techniques up to this point of the paper have been for
simple 1D or 2D convolution with a size $3$ or $3 \times 3$ matrix
respectively. However, an important feature of convolution in deep
neural networks is multiple channels. Convolution inputs and kernels
in many of the best known DNNs, such as GoogLeNet \cite{Szegedy15} or
ResNet \cite{He16} typically have something between 3 and 1024
channels. However, the number of channels is a parameter selected by
the designer of the neural network, and there is no upper limit on the
number of channels used.

When performing convolution, a separate 1D or 2D convolution is
performed at each channel, and then the results of each separate
convolution is summed with the corresponding values in the other
channels see figure \ref{fig:DNN-diagram}.  The obvious way to
implement summation across channels is to perform the complete
convolution separately on each channel, and sum the results. However,
this would require that the post-processing linear transform is
applied on each channel, which is a relatively expensive
operation. Lavin and Gray \cite{Lavin16} observed that the items can
be summed \textit{before} the linear transform, so that the
post-processing step need be applied only to the sum.  Therefore, in
order to perform convolution over multiple channels we have to sum up
the results of pairwise multiplication  (Hadamard product) before we apply the
transposition represented by $A^T$ matrix. The convolution is computed
according the following formulas:
$$A^T(\sum_{channels}(Gh \odot B^Tx))$$
$$A^T(\sum_{channels}(GHG^T \odot B^TXB))A$$

As a result, the FP error from DNN convolution is not just the error
of the 1D or 2D convolution, but also the error from summing across
channels. This is important for two reasons. First, if using Toom-Cook
convolution increases the numerical error. The impact of summation
over channels in error formulas (\ref{eq:channel_error_1D}) and
(\ref{eq:channel_error_2D}) is represented by $\lambda$.

Second, there are well-known techniques for reducing the
error from summation. If we reduce the error of summation, this may
offset some part of the loss of accuracy arising from Toom-Cook
convolution.

In section \ref{sec:mixed} we proposed a mixed precision algorithm
that does pre-/post-processing in higher precision. However, this is
not a suitable approach to increase the accuracy of the summation
across channels. The summation across channels is in the inner-most
loop of DNN convolution, so we cannot afford to double its cost.
We instead propose using the well-known pairwise summation algorithm
\cite{Knuth98} for summation across channels.

When summing $n$ FP inputs, the worst-case error from simply
accumulating to a single variable is $O(n)$. In contrast, the pairwise
summation algorithm has a worst-case error of just $O(\log_2n)$
\cite{Knuth98}.  Given our existing error bound for Toom-Cook
convolution with multiple channels, we can formulate the effect of
using pairwise summation rather than linear summation.

\begin{corollary}
	The error for 1D convolution based on theorems (\ref{theorem:TC-1D-error}) and (\ref{theorem:TC-2D-error}), using linear summation across channels is:
	\begin{equation}
	\label{eq:channels_linear}
	\|\hat{s}-s\|_1 \leq \|A^T\|_1\thinspace\|G\|_F\thinspace\|B^T\|_F\thinspace\|h\|_F\thinspace\|x\|_F\left(\alpha^{(n)}+\beta^{(n)}+\gamma^{(n_h)}+C\right)\varepsilon+O(\varepsilon^2)
	\end{equation}
	
	For pairwise summation across channels, the corresponding error is:
	\begin{equation}
	\label{eq:channels_pairwise}
	\begin{split} 
	&\|\hat{s}-s\|_1 \leq\\ &\leq \|A^T\|_1\thinspace\|G\|_F\thinspace\|B^T\|_F\thinspace\|h\|_2\thinspace\|x\|_2\left(\alpha^{(n)}+\beta^{(n)}+\gamma^{(n_h)}+\lfloor log(C) \rfloor+2\right)\varepsilon+O(\varepsilon^2)
	\end{split}
	\end{equation}
\end{corollary}

As we can observe comparing formulas
(\ref{eq:channels_linear}) and (\ref{eq:channels_pairwise}) we have smaller
overall error when we use the pairwise summation over channels than
for linear summation.

Tables \ref{tab:channels1Dfl}, \ref{tab:channels2Dfl},
\ref{tab:channel1Dfld}, \ref{tab:channels2Dfld} present measured
errors for Toom-Cook convolution with just a single channel, with 32
channels, and with 64 channels. We see that the error per output value
for 64 input channels is much larger than the error for just a single
input channel. But the error is around 2--20 times larger, not 64
times larger. The reason is that when these values are summed, some of
the errors cancel one another.

Our results presented in tables \ref{tab:channels1Dfl},
\ref{tab:channels2Dfl}
show that pairwise summation can reduce the total FP error by around
20\%--40\%.  Similar tests presented in tables
\ref{tab:channel1Dfld},\ref{tab:channels2Dfld} show that when pairwise
summation across channels is used with mixed-precision transforms, the
improvement compared to mixed-precision transforms alone is 25\%-45\%.
Using both proposed methods: mixed precision and pairwise summation
(\ref{tab:Ratio}) give us an improvement in accuracy of around $50\%$ in
both one- and two-dimensional computations.

\begin{table}[h]
	\caption{Toom-Cook over multiple channels in $FP32$ --- error per single output point for 1 dimension convolution. Columns "ratio in $\%$" present the ratio of error per single output point get with pairwise summation and error per single output point with liner summation in $\%$.}
	
	\label{tab:channels1Dfl}
	\begin{minipage}{\columnwidth}
		\begin{center}
			\begin{tabular}{cccccccc}
				\toprule
				Out & 1       & 32      & 32 channels    & ratio   & 64      & 64 channels   & ratio \\
				size   & channel & channels & pairwise sum  & in $\%$ & channels & pairwise sum & in $\%$ \\
				\hline
				1 & 1.75E-08 & 2.74E-07 & 1.90E-07 & $69\%$ &5.12E-07 & 2.87E-07 & 56$\%$ \\
				2 & 2.45E-08 & 3.80E-07 & 2.71E-07 & $71\%$ &7.03E-07 & 4.00E-07 & 57$\%$ \\
				3 & 5.19E-08 & 7.08E-07 & 5.11E-07 & $72\%$ &1.28E-06 & 7.59E-07 & 59$\%$ \\
				4 & 6.92E-08 & 8.35E-07 & 6.17E-07 & $74\%$ &1.48E-06 & 9.18E-07 & 62$\%$ \\
				5 & 9.35E-08 & 1.09E-06 & 8.35E-07 & $77\%$ &2.00E-06 & 1.24E-06 & 62$\%$ \\
				6 & 1.15E-07 & 1.31E-06 & 9.79E-07 & $75\%$ &2.34E-06 & 1.47E-06 & 63$\%$ \\
				7 & 2.34E-07 & 2.90E-06 & 2.16E-06 & $74\%$ &5.21E-06 & 3.20E-06 & 61$\%$ \\
			\end{tabular}
		\end{center}
		\bigskip\centering
	\end{minipage}
\end{table}%

\begin{table}[h]
	\caption{Toom-Cook over multiple channels in $FP32$ - error per single output point for two-dimensional convolution. Columns "ratio in $\%$" present the ratio of error per single output point get with pairwise summation and error per single output point with linear summation in $\%$.}
	
	\label{tab:channels2Dfl}
	\begin{minipage}{\columnwidth}
		\begin{center}
			\begin{tabular}{cccccccc}
				\toprule
				Out & 1       & 32      & 32 channels   & ratio   & 64      & 64 channels & ratio \\
				size   & channel & channels & pairwise sum & in $\%$ & channels & pairwise sum & in $\%$\\
				\hline
				$1 \times 1$ & 4.63E-08 & 5.25E-07 & 3.95E-07 & $75\%$ & 9.44E-07 & 5.83E-07 & 62$\%$ \\
				$2 \times 2$ & 7.65E-08 & 9.05E-07 & 6.47E-07 & $71\%$ & 1.65E-06 & 9.59E-07 & 58$\%$\\
				$3 \times 3$ & 2.35E-07 & 2.87E-06 & 2.09E-06 & $73\%$ & 5.33E-06 & 3.11E-06 & 58$\%$\\
				$4 \times 4$ & 3.29E-07 & 3.60E-06 & 2.70E-06 & $75\%$ & 6.56E-06 & 3.98E-06 & 61$\%$\\
				$5 \times 5$ & 6.81E-07 & 7.78E-06 & 5.71E-06 & $73\%$ & 1.41E-05 & 8.57E-06 & 61$\%$\\
				$6 \times 6$ & 8.79E-07 & 9.48E-06 & 7.12E-06 & $75\%$ & 1.71E-05 & 1.04E-05 & 61$\%$\\
				$7 \times 7$ & 2.43E-06 & 4.66E-05 & 3.41E-05 & $73\%$ & 8.41E-05 & 5.09E-05 & 61$\%$\\
			\end{tabular}
		\end{center}
		\bigskip\centering
	\end{minipage}
\end{table}%

\begin{table}[h]
	\caption{Toom-Cook computation in $FP32$ with transforms in $FP64$- error per single output point for 1 dimension convolution. Columns "ratio in $\%$" present the ratio of error per single output point get with pairwise summation and error per single output point with liner summation in $\%$.}
	
	\label{tab:channel1Dfld}
	\begin{minipage}{\columnwidth}
		\begin{center}
			\begin{tabular}{cccccccc}
				\toprule
				Out & 1       & 32      & 32 channels    & ratio   & 64      & 64 channels & ratio \\
				size   & channel & channels & pairwise sum  & in $\%$ & channels & pairwise sum & in $\%$\\
				\hline
				1 & 1.75E-08 & 2.73E-07 & 1.90E-07 & $70\%$ & 5.15E-07 & 2.88E-07 & $56\%$ \\
				2 & 1.87E-08 & 3.60E-07 & 2.31E-07 & $64\%$ & 6.73E-07 & 3.58E-07 & $53\%$ \\
				3 & 3.66E-08 & 6.50E-07 & 4.39E-07 & $68\%$ & 1.20E-06 & 6.72E-07 & $56\%$ \\
				4 & 4.41E-08 & 7.45E-07 & 5.16E-07 & $69\%$ & 1.41E-06 & 7.86E-07 & $56\%$ \\
				5 & 6.09E-08 & 1.00E-06 & 6.98E-07 & $70\%$ & 1.92E-06 & 1.06E-06 & $55\%$ \\
				6 & 6.97E-08 & 1.17E-06 & 7.90E-07 & $68\%$ & 2.18E-06 & 1.20E-06 & 55$\%$\\
				7 & 1.55E-07 & 2.60E-06 & 1.80E-06 & $69\%$ & 4.91E-06 & 2.75E-06 & 56$\%$\\
			\end{tabular}
		\end{center}
	\end{minipage}
\end{table}%

\begin{table}[h]
	\caption{Toom-Cook computation in $FP32$ with transforms in $FP64$- error per single output point for two-dimensional convolution. Columns "ratio in $\%$" present the ratio of error per single output point get with pairwise summation and error per single output point with liner summation in $\%$.}
	
	\label{tab:channels2Dfld}
	\begin{minipage}{\columnwidth}
		\begin{center}
			\begin{tabular}{cccccccc}
				\toprule
				Out & 1       & 32      & 32 channels    & ratio  & 64      & 64 channels & ratio \\
				size   & channel & channels & pairwise sum  & in $\%$& channels & pairwise sum & in $\%$\\
				\hline
				$1 \times 1$ & 4.63E-08 & 5.25E-07 & 3.95E-07 & $75\%$ & 9.48E-07 & 5.85E-07 & $62\%$ \\
				$2 \times 2$ & 5.27E-08 & 8.51E-07 & 5.54E-07 & $65\%$ & 1.59E-06 & 8.48E-07 & $53\%$ \\
				$3 \times 3$ & 1.62E-07 & 2.70E-06 & 1.80E-06 & $67\%$ & 5.13E-06 & 2.75E-06 & $54\%$ \\
				$4 \times 4$ & 2.14E-07 & 3.60E-06 & 2.36E-06 & $66\%$ & 6.68E-06 & 3.61E-06 & $54\%$ \\
				$5 \times 5$ & 3.69E-07 & 6.06E-06 & 4.07E-06 & $67\%$ & 1.14E-05 & 6.17E-06 & $54\%$ \\
				$6 \times 6$ & 5.18E-07 & 8.48E-06 & 5.64E-06 & $67\%$ & 1.59E-05 & 8.53E-06 & $54\%$\\
				$7 \times 7$ & 3.39E-06 & 4.21E-05 & 2.91E-05 & $69\%$ & 8.03E-05 & 4.34E-05 & $54\%$
			\end{tabular}
		\end{center}
	\end{minipage}
\end{table}%

\begin{table}[h]
	\caption{Ratio of Toom-Cook in $FP32$ with linear summation over the channels and in mixed precision with pairwise summation over the channels - error for single output point in $1$ and $2$ dimensions.}
	\label{tab:Ratio}
	\begin{minipage}{\columnwidth}
		\begin{center}
			\begin{tabular}{cccccc}
				\toprule
				Output size & 32       & 64       & Output size & 32       & 64 \\
				1D          & channels & channels & 2D          & channels & channels \\
				\hline
				1 & $69\%$ & $56\%$ & $1 \times 1$ & $75\%$ & $62\%$ \\
				2 & $61\%$ & $51\%$ & $2 \times 2$ & $61\%$ & $51\%$ \\
				3 & $62\%$ & $53\%$ & $3 \times 3$ & $63\%$ & $52\%$ \\
				4 & $62\%$ & $53\%$ & $4 \times 4$ & $66\%$ & $55\%$ \\
				5 & $65\%$ & $53\%$ & $5 \times 5$ & $52\%$ & $44\%$ \\
				6 & $60\%$ & $51\%$ & $6 \times 6$ & $59\%$ & $50\%$ \\
				7 & $62\%$ & $53\%$ & $7 \times 7$ & $62\%$ & $52\%$
			\end{tabular}
		\end{center}
		\bigskip\centering
	\end{minipage}
\end{table}

\section{Related work}
Lavin and Gray \cite{Lavin16} wrote the seminal paper on applying
Winograd convolution to deep neural networks. They showed how to apply
2D versions of these algorithms to DNN convolution with multiple input
and output channels, and how to amortize the cost of
pre-/post-processing over many convolutions.  Although they used the
Toom-Cook method to generate their core convolution algorithms, they
refered to it as Winograd convolution, and that has become the
accepted term in the DNN literature.

By far the closest existing work to ours is from Vincent et
al. \cite{Vincent17}. They propose to scale convolution matrices to
reduce the condition number of the Vandermonde matrices. They
demonstrate that this approach can reduce the error number in exactly
one case: convolving a $5 \times 5$ kernel to create a $9 \times 9$
output block. Further they showed that this improved matrix could be
used to successfully for training a DNN. However, they did not provide
a method for choosing good scaling factors. Our approach to reducing
FP error is equally empirical, but we focus on constructing good
convolution matrices rather than improving them after construction. We
measured the error for our $5 \times 5$, $9 \times 9$ (with 13
\textit{points}) convolution matrices and compared it with Vincent et
al.'s solution. We found that our convolution matrices yield an error
that is around 45\% lower (better) than Vincent et al.'s.

The idea of applying the Toom-Cook algorithm to compute convolution
was investigated and in great detail by Shmuel Winograd
\cite{Winograd80}. He focused on the low complexity of Toom-Cook
convolution, and proved that it is optimal with respect to the number
of general multiplications. Winograd developed his own method of
generating short convolution algorithms based on the Chinese remainder
theorem.  Winograd's method can create a much larger set of algorithms
than Toom-Cook, including algorithms that are not optimal with respect
to general mutliplications.

To the best of our knowledge, this paper is the first that presents a theoretical analysis of the numerical error of Toom-Cook convolution.
We demonstrate that the algorithm is unstable because the properties of algorithm parameters matrices we use \linebreak ($G$, $B^T$, $A^T$).
We formulate the boundaries of the FP errors for one- and two-dimensional kernels and inputs so we can resonable choose what we should
focus on to improve the accuracy. We formulated the error bounds for Toom-Cook convolution using similar techniqes to those used for another bilinear problem: the fast matrix multiplication, error estimation by Bini and Lotti \cite{Bini80}, Demmel et al. \cite{Demmel07} and \cite{Ballard16}. We show that algorithm is unstable and how the error depends on each component.

As we can see from our errors formulation, the stability of the Toom-Cook convolution depends directly on the values of the matrices $G$, $A^T$ and $B^T$ not only on input and kernel values and sizes. While it has been empirically observed and theoreticaly proven that the condition number of square Vandermonde matrices containing real (not complex) points increases exponentially, to our knowledge there is no theoretically-sound method for choosing the best points. There are some more specific boudaries for particular sets of points, i.e. harmonic, equidistance, positive, in the range $(-1,1)$ \cite{Higham02} and complex points \cite{Pan16}. The way we choose points in our tests allow us to obtain sets of good points for specific input and kernel sizes, but points were find to be good empirically do not follow any of these simple patters. In general, minimizing the numerator and denominator works well, but it is not always the best approach. The FP representation \cite{Goldberg91} \cite{Higham02} and symmetry of reciprocal and inverse points \cite{Bodrato07} \cite{Gautschi90} matters as well.

In our search for the best points, we studied a wide range of literature on the conditiining of Vandermonde matrices dating from the last $40+$ years \cite{Pan16} \cite{Gautschi90} \cite{Gautschi74/75} \cite{Higham02} and Toom-Cook algorithm \cite{Bodrato07} \cite {Selesnick94}. We took under consideration all theoretical results we found while developing our strategy for finding good points. Since there is not any clear pattern of points to choose, the lack of theoretical background on Vandermonde inverse matrix norm did not allow us to present any more advanced analysis.

Some work on Toom-Cook optimality was done by M. Bodrato
\cite{Bodrato07}.  He focused on the optimality of this algorithm
applied to the polynomial multiplication problem, as measured by the
number of operations required. Improving the floating numerical
accuracy of the result was not a goal of Bodrato's work, and no data
data is provided of the effect of the proposed techniques on numerical
accuracy. In contrast, our work studied Toom-Cook algorithm
application for the DNN convolution problem; we consider a much bigger
variety of input sizes and additional factors that have an impact on
accuracy like FP precision. However, as we have shown,
reducing the number of operations required for the
pre-/post-processing steps can improve numerical accuracy. Just as we
found that symmetric points can improve numerical accuracy, Bodrato
found that such points could reduce the number of required operations.

Note that our focus is on reducing Toom-Cook FP error to allow larger
outblock block sizes and thus fewer \textit{general
	multiplications}. However, error analysis might also be used for
other purposes such as identifying when the training error has become
smaller than the FP error, and that therefore training can be
terminated early. We leave such additional uses of error analysis to
future work.

\section{Conclusions}
We present an analysis of 1D and 2D Toom-Cook convolution with
multiple channels for DNN convolution. We identify and formalize the
error terms for Toom-Cook convolution, and prove that the error bound
grows at least exponentially with the size of the convolution.  This
result is supported by an analysis of conditioning, which shows that
the condition number of the convolution with respect to the norm grows
exponentially with convolution size.

We formally analyse the error bound for the ``modified'' Toom-Cook
algorithm, and prove that the error is close to that of the
non-modified algorithm operating on an input one element smaller.  We
observe empirically that using modified version reduces the error by
$20\%$ to over $70\%$, with no additional computation cost.

We observe that the order of point selection impacts the accuracy of
the algorithm. We propose a canonical evaluation order based on
Huffman trees. This fixes the order of evaluation, and empirically
reduces the error by a further 12\%-14\% \textit{in addition to} the
improvement from using the modified algorithm, again with no
additional computation cost.

There is no existing, widely accepted strategy for selecting
points. We observe that for DNNs there is a relatively small number of
important sizes, and we search empirically for good point selections
for those sizes. We identify four key criteria for good point
selections: (1) few significant mantissa bits, (2) positive/negative
point symmetry, (3) reciprocal point symmetry, and (4) subtractions of
points leading to few significant mantissa bits. For important
convolution sizes for DNNs, our empirically selected points yield much
better accuracy than the Chebyshev nodes. The Chebyshev nodes fail
to meet our criteria (1), (3) and (4) which makes them relatively
poor choices for the small convolutions found in DNNs.

We also proposed a mixed precision approach where the transforms are
computed in double precision, while the remaining inner loops are
computed in $FP32$. We found that perfoming pre/post processing
transforms in $FP64$ decreases error typically by around one third in
our experiments. We also empirically investigated summation across
channels using pairwise summation and found that this reduces the
error by around 20\% to 50\%. Unlike the other methods we
investigated, mixed precision transforms and pairwise summation impose
an additional computation cost.

Using our point selections and techniques for improving FP accuracy,
we can reduce the error between $2\times$ and orders of magnitude,
when compared with the Chebyshev nodes. Each approximately $2\times$
reduction in the error allows the output block size dimension to be
increased by around one. Whereas current implementations of Winograd
convolution for DNNs typically use output block sizes of $2 \times 2$
or $3 \times 3$, our methods allow larger block sizes, with a
resulting reduction in arithmetic operations (see Table
\ref{tab:mult}). This will allow faster DNN training and inference,
which is particularly valuable for resource-constrained mobile and
embedded systems.

\appendix
\section*{APPENDIX}
\input{normcondestimate.tex}
\pagebreak
\section{Proof of the theorem of two-dimensional Toom-Cook error bounds}
\label{sec:2dim}

\subsection{Two-dimensional Toom-Cook convolution error}

\begin{theorem}
	\label{theorem:TC-2D-error}
	Error for two-dimensional Toom-Cook convolution computation satisfies the componentwise bound equal to:
	\begin{equation}
	\label{eq:TC-2D-element}
	\begin{split} 
	|\hat{S}-S| &\leq \\
	\leq &|A^T|\left(|G||H||G^T| \odot |B^T||X||B|\right)|A|\left(2\alpha^{(n)}+2\beta^{(n)}+2\gamma^{(n_h)}+1\right)\varepsilon+\O(\varepsilon^2) \\
	\end{split}
	\end{equation}
	Error for two-dimensional Toom-Cook convolution computation satisfies the normwise bound equal to:
	\begin{equation}
	\label{eq:TC-2D-norm}
	\begin{split} 
	\|\hat{S}-S\|_1 &\leq \\
	& \leq \|A^T\|_1\thinspace\|G\|_F\thinspace\|H\|_F\thinspace\|G^T\|_F\thinspace\|B^T\|_F\thinspace\|X\|_F\thinspace\|B\|_F\thinspace\|A\|_1\thinspace R\varepsilon+O(\varepsilon^2)\;\;\;\;\;\;\;\;\;\\
	\mathrm{where} \; R& = 2\alpha^{(n)}+2\beta^{(n)}+2\gamma^{(n_h)}+1 \;\;\;\;\;\;\;\;\;\;\;\;\;\;\;\;\;\;\;\;\;\;\;\;
	\end{split}
	\end{equation}
	If we assume the same method of summation while matrix  and transpose matrix multiplication.\\ Where $\alpha^{(n)}$, $\beta^{(n)}$, $\gamma^{(n_h)}$ represents the error from multiplication by matrices $A^T$, $B^T$ and $G$ respectively.
\end{theorem}
\begin{proof}
	In computing two-dimensional convolution we use the feature of Kronecker product that means $vec(MXM^T) = (M \otimes M) vec(X)$. Notice that despite the both formulas are mathematically equivalent the result in FP arithmetic could be different because of more multiplication operations while computing $vec(MXM^T)$ than for $(M \otimes M) vec(X)$.\\
	We have $$|fl(fl(A)Xfl(A^T))|=|A|\thinspace|X|\thinspace|A^T|2\alpha^{(n)}\varepsilon+O(\varepsilon^2)$$. 
	Let put component function 
	$$S=g(H,X)=g_3(g_2(g_1^{H}(H),g_1^{X}(X)))$$   
	$g(H,X):\mathbb{R}^{n_h \times n_h}\times \mathbb{R}^{n \times n} \rightarrow \mathbb{R}^{n_o \times n_o}$, $\;\;g(H,X)=A^T\left(GHG^T \odot B^TXB\right)A$, \\
	$g_3:\mathbb{R}^{n\times n} \rightarrow \mathbb{R}^{n_o \times n_o}$, $\;\;g_3(M)=A^Tm$,\\ $g_2:\mathbb{R}^{n\times n}\times \mathbb{R}^{n\times n}\rightarrow \mathbb{R}^{n\times n}$, 
	$\;\;g_2(M,N)=M \odot N$,\\
	$g_1:\mathbb{R}^{n_h \times n_h}\rightarrow \mathbb{R}^{n\times n}$, $\;\;g_1^{H}(M)=GMG^T$ and $g_1^{X}:\mathbb{R}^{n\times n}\rightarrow \mathbb{R}^{n\times n}$.\\ Similarly as for one-dimensional convolution we have:
	% \begin{align*}
	$$|vec(\hat{S})-vec(S)|\leq|J_3|\thinspace|J_2|\thinspace|\Delta b_2|+|J_2|\thinspace|\Delta b_3|+|\Delta b_4|=$$
	$$=|A^T \otimes A^T|\left[Diag((|B^T|\otimes |B^T|)vec(|X|)),Diag((|G|\otimes |G|)vec(|H|))\right]|$$
	$$|\left[
	\begin{array}{l}
	\left(|G|\otimes |G|\right)vec\left(|H|\right)2\gamma^{(n_h)} \varepsilon +O(\varepsilon^2)\\
	\left(|B^T|\otimes |B^T|)vec(|X|\right)|2\beta^{(n)} \varepsilon + O(\varepsilon^2)
	\end{array}
	\right]| + $$
	$$+\left(|A^T| \otimes |A^T|\right)\left(\left(|G|\otimes |G|\right)vec\left(|H|\right) \odot \left(|B^T|\otimes |B^T|\right)vec\left(|X|\right)\right)\varepsilon+O(\varepsilon^2)+$$
	$$\left(|A^T|\otimes |A^T|\right)\left(\left(|G|\otimes |G|\right)vec\left(|H|\right) \odot \left(|B^T|\otimes |B^T|\right)vec\left(|X|\right)\right)2\alpha^{(n)}\varepsilon+O(\varepsilon^2)=$$
	$$=\left(|A^T|\otimes |A^T|\right)\left(\left(|G|\otimes|G|\right)vec\left(|H|\right)\odot \left(|B^T|\otimes |B^T|\right)vec\left(|X|\right)\right)\left(2\gamma^{(n_h)}+2\beta^{(n)}\right)\varepsilon \thinspace+$$ 
	%  +O(\varepsilon^2)+$$
	$$+\left(|A^T|\otimes |A^T|\right)\left(\left(|G|\otimes|G|\right)vec\left(|H|\right)\odot \left(|B^T|\otimes |B^T|\right)vec\left(|X|\right)\right)\left(2\alpha^{(n)}+1\right)\varepsilon+O(\varepsilon^2)=$$
	$$=\left(|A^T| \otimes |A^T|\right)\left(|G|\otimes |G|\right)vec\left(|H|\right) \odot \left(|B^T|\otimes |B^T|\right)vec\left(|X|\right)|\thinspace R\varepsilon+\thinspace O(\varepsilon^2)=$$
	$$R=2\gamma^{(n_h)}+2\beta^{(n)}+2\alpha^{(n)}+1$$
	% \end{align*}
	Changing vectors again to matrices we obtain:
	$$|\hat{S}-S|\leq|A^T|\left(|G|\thinspace|H|\thinspace|G^T|\odot |B^T|\thinspace|X|\thinspace|B|\right)|A|\left(2\alpha^{(n)}+2\beta^{(n)}+2\gamma^{(n_h)}+1\right)\varepsilon+O(\varepsilon^2)$$
	We do the normwise error estimationin in the same way as for one-dimensional convolution, as the matrix consistency and Buniakowski-Schwartz inequality holds both for vectors and matrices.
	$$\|\hat{S}-S\|_1 \leq \|A^T\left(GHG^T\odot B^TXB\right)A\left(2\alpha^{(n)}+2\beta^{(n)}+2\gamma^{(n_h)}+1\right)\varepsilon+O(\varepsilon^2)\| \leq$$
	$$ \leq \|A^T\left(GHG^T\odot B^TXB\right)A\|_1\left(2\alpha^{(n)}+2\beta^{(n)}+2\gamma^{(n_h)}+1\right)\varepsilon+O(\varepsilon^2)$$
	From norm $\|\cdot\|_1$ consistency:
	$$\|\hat{S}-S\|_1 \leq \|A^T\|_1\|\left(GHG^T\right)\odot\left(B^TXB\right)\|_1\|A\|_1\left(2\alpha^{(n)}+2\beta^{(n)}+2\gamma^{(n_h)}+1\right)\varepsilon+O(\varepsilon^2)$$
	Applying Buniakowski-Schwartz inequality for Hadamard product:
	$$\|\hat{S}-S\|_1 \leq \|A^T\|_1\thinspace\|GHG^T\|_F\thinspace\|B^TXB)\|_F\thinspace\|A\|_1\left(2\alpha^{(n)}+2\beta^{(n)}+2\gamma^{(n_h)}+1\right)\varepsilon+O(\varepsilon^2)$$
	Finally from norm equivalency:
	$$\|\hat{S}-S\|_1 \leq \|A^T\|_1\thinspace\|G\|_F\thinspace\|H\|_F\thinspace\|G^T\|_F\thinspace\|B^T\|_F\thinspace\|X\|_F\thinspace\|B\|_F\thinspace\|A\|_1\thinspace R\varepsilon+O(\varepsilon^2)$$
	$$R=2\alpha^{(n)}+2\beta^{(n)}+2\gamma^{(n_h)}+1$$
\end{proof}

\section{Modified Toom-Cook convolution error analysis}
\label{sec:modWinograd_error}

\begin{theorem}
	The componentwise error for one-dimensional modified Toom-Cook for $q$th element of output is bounded by:
	\begin{equation}
	\label{eq:ModWinograd-1D-element-appendix}
	\begin{split}
	&|\hat{s}^{m(n)}_q-s^{m(n)}_q| = \\
	&=|{A^T_{q:}}^{(n-1)}|\thinspace\left(|G^{(n-1)}|\thinspace
	|h|\odot|B^{(n-1)T}|\thinspace|x^{(n-1)}|\right)\left(\gamma^{(n_h)}+\beta^{(n-1)}+\alpha^{(n-1)}+1\right)\varepsilon + O(\varepsilon^2) \\ 
	&\mathrm{for} \; q=1,..,n_o-1\\
	& \\
	&|{\hat{s}_{n_o}}^{(m(n))}-s^{m(n)}_{n_o}| \leq \\ &\leq |{A^T_{q:}}^{(n-1)}|\left(|G^{(n-1)}|\thinspace|h|\odot|B^{(n-1)T}|\thinspace|x^{(n-1)}|+
	|h_{n_h}|\thinspace|{B_{n:}}^{m(n)T}|\thinspace|x|\right)\\
	&\;\;\;\;\;\;\;\;\;\left(max\left\{\left(\gamma^{(n_h)}+\beta^{(n-1)}+\alpha^{(n-1)}+1\right), \left(\beta+1\right)\right\}+1\right)\varepsilon + O(\varepsilon^2)\\
	&\mathrm{for} \; q=n_o
	\end{split}
	\end{equation}
\end{theorem}
\begin{proof}
	We denote matrices constructed for Toom-Cook algorithm of input size $n$ as $G^{(n)}$, $A^{(n)T}$ and $B^{(n)T}$ and for modified Toom-Cook algorithm of input size $n$ as $G^{m(n)}$, $A^{m(n)T}$ and $B^{m(n)T}$.
	
	As we can see from the modified Toom-Cook algorithm definition we solve the problem of size $n$ by solving problem of size $n-1$ by Toom-Cook algorithm and proper modify the last element of output.
	Thus is obvious that error boundary of the first $n_o-1$ output points in modified Toom-Cook convolution is equal to the boundary of Toom-Cook convolution for input size $n-1$.
	The computation of the last $n_o$th output point include two parts: we compute the partial result using Toom-Cook convolution for input $n-1$ and add the missing value.
	${s_{n_o}}^{m(n)}={s^{(n-1)}_{n_o}}+h_{n_h}B^{m(n)_{n:}}x$
	So:
	$$|{\hat{s}_{n_o}}^{m(n)}-s_{n_o}^{m(n)}| \leq$$ %|fl({\hat{s}_{n_o})}^{(n-1)}+fl(h_{n_h}fl({B_{n:}}^{m(n)T}x))-{s_{n_o}}^{(n-1)}-h_{n_h}{B_{n:}}^{m(n)T}x| \leq$$
	$$|{s_{n_o}}^{(n-1)}|\left(\gamma^{(n_h)}+\beta^{(n-1)}+\alpha^{(n-1)}+1\right)\varepsilon+O(\varepsilon^2) + |h_{n_h}|\thinspace|{B_{n:}}^{m(n)T}|\thinspace|x|(\beta^{(n)}+1)\varepsilon+O(\varepsilon^2)=$$
	$$=\left(|{s_{n_o}}^{(n-1)}|+|h_{n_h}|\thinspace|{B_{n:}}^{m(n)T}||x|\right)max\left\{\gamma^{(n_h)}+\beta^{(n-1)}+\alpha^{(n-1)}+1,\beta^{n}+1\right\}\varepsilon+O(\varepsilon^2)$$
\end{proof}

For DNN convolution the kernel size $n_h$ is almost always greater than or equal to $3$ and $n_h \leq n$. If we compute the convolution then from dot product error analysis we know that $\gamma^{(n_h)} \geq n_h-1$, $\beta^{(n-1)} \geq n-2$ and $\alpha^{(n-1)} \geq n-2$. In this case $\gamma^{(n_h)}+\beta^{(n-1)}+\alpha^{(n-1)}+1 \geq n_h+2n-4$, while $\beta^{(n)}+1 \leq n+2$.
When $n_h \geq 3$ the maximum in formula \ref{eq:ModWinograd-1D-element-appendix} is always equal to $\gamma^{(n_h)}+\beta^{(n-1)}+\alpha^{(n-1)}+1$. So we can formulate following corollary:

\begin{corollary}
	\label{cor:cor1}
	If we use the kernel size $n_h \geq 3$ computing for 1D convolution using modified Toom-Cook algorithm then the componentwise error for last output element is bounded as follows:
	\begin{equation}
	\label{eq:modT-C-1D-cor1}
	\begin{split}
	&|{\hat{s}_{n_o}}^{(m(n))}-s^{m(n)}_{n_o}| \leq \\ &\leq |{A^T_{q:}}^{(n-1)}|\left(|G^{(n-1)}|\thinspace|h|\odot|B^{(n-1)T}|\thinspace|x^{(n-1)}|+
	|h_{n_h}|\thinspace|{B_{n:}}^{m(n)T}|\thinspace|x|\right)R\varepsilon + O(\varepsilon^2)
	\end{split}
	\end{equation}
	where $R=\left(\gamma^{(n_h)}+\beta^{(n-1)}+\alpha^{(n-1)}+1\right)$
\end{corollary}
Assuming any elements in matrices $G^{m(n)}$, $B^{m(n)T}$ and $A^{m(n)T}$ and linear summation while compute the dot product we have following boundaries:

\begin{corollary}
	\label{cor:cor2}
	\begin{equation}
	\label{eq:modT-C-1D-cor1}
	\begin{split}
	&|{\hat{s}_{n_o}}^{(m(n))}-s^{m(n)}_{n_o}| \leq \\ &\leq(|{A^T_{q:}}^{(n-1)}|(|G^{(n-1)}||h|\odot|B^{(n-1)T}||x^{(n-1)}|+
	|h_{n_h}||{B_{n:}}^{m(n)T}||x|)(n_h+2n+2)\varepsilon + O(\varepsilon^2)\nonumber
	\end{split}
	\end{equation}
\end{corollary}

\section{Tests results for Chebyshev nodes}
\label{sec:Chebyshev}
The table below presented the error of convolution computations for different sizes with Chebyshev points and the points we found the best. The column "Ratio" presents how many times the error for Chebyshev points is bigger then for points we found.

\pagebreak

\begin{table}[t]
	\begin{tabular}{c|cccccc}
		& 1D & 1D & 1D & 2D & 2D & 2D \\
		n & Our points & Chebyshev & Ratio & Our points & Chebyshev & Ratio \\
		\toprule
		4 & 2.49E-08 & 3.51E-08 & 1.41 & 7.65E-08 & 1.01E-07 & 1.32 \\
		5 & 5.21E-08 & 5.535E-08 & 1.06 & 2.46E-07 & 2.12E-07 & 8.62 \\
		6 & 7.32E-08 & 1.04E-07 & 1.42 & 3.49E-07 & 5.65E-07 & 1.62 \\
		7 & 1.00E-07 & 1.68E-07 & 1.68 & 7.11E-07 & 1.70E-06 & 2.39 \\
		8 & 1.23E-07 & 3.78E-07 & 3.07 & 9.21E-07 & 6.50E-06 & 7.06 \\
		9 & 2.59E-07 & 6.76E-07 & 2.61 & 4.07E-06 & 2.59E-05 & 6.36 \\
		10 & 3.87E-07 & 1.48E-06 & 3.82 & 8.23E-06 & 1.09E-04 & 1.32E+01 \\
		11 & 6.62E-07 & 3.18E-06 & 4.80 & 2.53E-05 & 4.84E-04 & 1.91E+01 \\
		12 & 8.42E-07 & 7.46E-06 & 8.86 & 3.67E-05 & 2.18E-03 & 5.94E+01 \\
		13 & 1.56E-06 & 1.53E-05 & 9.81 & 1.24E-04 & 1.00E-02 & 8.06E+01 \\
		14 & 2.10E-06 & 3.21E-05 & 15.29 & 2.32E-04 & 4.72E-02 & 2.03E+02 \\
		15 & 3.91E-06 & 7.12E-05 & 18.21 & 6.35E-04 & 2.31E-01 & 3.64E+02 \\
		16 & 5.06E-06 & 1.56E-04 & 30.83 & 1.04E-03 & 1.10 & 1.06E+03 \\
		17 & 1.68E-05 & 3.53E-04 & 21.01 & 1.31E-02 & 5.43 & 4.15E+02 \\
		18 & 2.36E-05 & 8.03E-04 & 34.03 & 0.24E-02 & 26.91 & 1.12E+04
	\end{tabular}
	\caption{Error for one- and two-dimensional convolution for Chebyshev points and points we found. Column "Ratio" present how many times the error for Chebyshev points is bigger then for points we found.}
\end{table}

	\section*{Acknowledgements}
	This work was supported by Science Foundation Ireland grant 12/IA/1381, and in part by Science Foundation Ireland grant 13/RC/2094 to Lero -- the Irish Software Research Centre (www.lero.ie).

\bibliographystyle{amsplain}
\bibliography{references_arxiv}

\end{document}

%% file: normcondestimate.tex
\section{Estimate of norm and conditioning}
\label{section.AppNormCond}
In this appendix, we provide some estimates of the norm and conditioning of the product $W(h,x) = A^{T}\left( Gh\odot B^{T}x \right)$ in terms of an expression involving only the
matrices $A$, $G$, and $B$ and one involving $x$ and $h$.  To do this, we reformulate this Hadamard product using a special product of two matrice called the Khatri-Rao 
product, which can be thought of as a specific block analog of the Hadamard product; see \cite{GV.2013}.
\begin{definition}
Let $C,F$ be $m\times n$ block matrices with the structure
\begin{equation}\nonumber
	C = \left( C_{ij} \right)_{ij} \qquad\mbox{and}\qquad F = \left( F_{ij} \right)_{ij}
\end{equation}
with block sizes $C_{ij},F_{ij}\in\mathbb{R}^{s_{i}\times s_{j}}$. The Khatri-Rao product is the blockwise Kronecker product defined by
$C \otimes_{KR} F = \left( C_{ij} \otimes F_{ij} \right)_{ij}$.
\end{definition}
The Khatri-Rao product can be used to express the Hadamard product $Gh\odot B^{T}x$ as a matrix-vector product, whereby we use block both 
matrices by row.  That means the $C \otimes_{KR} F$ here denotes the matrix whose rows consist of the Kronecker products of the rows of $C$ and $F$.
\begin{theorem}
The Hadamard product $Gh\odot B^{T}x$ admits the expression
\begin{equation}\nonumber
	Gh\odot B^{T}x = \left( B^{T}\otimes_{KR}G \right) \cdot \left( x\otimes h \right).
\end{equation}
\end{theorem}
\begin{proof}
	Let 
	\begin{equation}\nonumber
		G = \begin{bmatrix}
		g_{1}^{T} \\ 
		g_{2}^{T} \\ 
		\vdots \\ 
		g_{nx}^{T}
		\end{bmatrix} \qquad\mbox{and}\qquad 
		B^{T} = \begin{bmatrix}
		b_{1}^{T} \\ 
		b_{2}^{T} \\ 
		\vdots \\ 
		b_{nx}^{T}
		\end{bmatrix} 
	\end{equation}
	with $b_{i}^{T} = \begin{bmatrix} b_{i1} & b_{i2} & \cdots & b_{i,n_{x}} \end{bmatrix} $.  Then the $i$th entry of 
	 $B^{T}h$ is $b_{i}^{T}x = \sum_{j=1}^{n_{x}}b_{ij}x_{j}$. and 
	 the $i$th entry of $Gh\odot B^{T}x$ can thus be written as 
	\begin{equation}\nonumber
		g_{i}^{T}h\cdot b_{i}^{T}x =  g_{i}^{T}h \cdot \sum_{j=1}^{n_{x}}b_{ij}x_{j}=  \sum_{j=1}^{n_{x}}b_{ij}g_{i}^{T}\left( x_{j} h\right),
	\end{equation}
	which can be written as the vector dot product
	\begin{equation}\nonumber
	g_{i}^{T}h\cdot b_{i}^{T}x = 	\begin{bmatrix} b_{i1} g_{i}^{T} & b_{i2} g_{i}^{T} & \cdots &  b_{i,n_{x}} g_{i}^{T} \end{bmatrix} \begin{bmatrix}x_{1}h\\x_{2}h \\ \vdots \\ x_{n_{x}}h\end{bmatrix} = \left(  b_{i}^{T} \otimes g_{i}^{T}
	\right)  \left( x \otimes h \right).
	\end{equation}
	This proves the result.
\end{proof}
We now calculate the condition number of $W(h,x)$.   Recall that for a continuous function $f(x)$ 
and a given vector norm $\left\| \cdot \right\|$, we can express the relative condition number in terms of the induced operator norm of its Jacobian $J_{f}(x)$
\begin{equation}\nonumber
	\kappa(x) = \dfrac{\left\| J_{f}(x) \right\| \cdot \left\| x \right\|}{\left\| f(x) \right\|};
\end{equation}
see, e.g., \cite{Trefethen.Bau.Numerical-Linear-Algebra.1997}. We can write the composition $W(h,x) = W(y)$ where $y:=y(h,x) = x\otimes h$.  Then we can use the chain rule to express the 
Jacobian of $W(h,x)$ as
\begin{equation}\nonumber
	J_{w}(h,x) =  A^{T}\left( B^{T}\otimes_{KR}G \right) J_{y}(x,h),
\end{equation}
where $J_{y}(h,x)$ is the Jacobian of $y(h,x)$.  Thus the condition number for $W(h,x)$ satisfies
\begin{equation}\label{eqn.WCond}
	\kappa_{W}(h,x) = \dfrac{\left\| A^{T}\left( B^{T}\otimes_{KR}G \right) J_{y}(x,h) \right\|\cdot \left\| x\otimes h \right\|}{\left\| A^{T}\left( B^{T}\otimes_{KR}G \right) \cdot \left( x\otimes h \right) \right\|}.
\end{equation}
We can then get an upper bound estimate for this condition number.  We begin by introducing some terminology.  If $C\in\mathbb{R}^{m\times m}$ is nonsingular, then
it is a well-known result that the condition number $\kappa_{2}(C)$ with respect to $\left\| \cdot \right\|_{2}$ satisfies $\kappa_{2}(C)=\dfrac{\sigma_{max}(C)}{\sigma_{min}(C)}$, where $\sigma_{max}(C),\sigma_{min}(C) > 0$ are, 
respectively, the largest and smallest singular value of $C$, which are guaranteed by the nonsingularity of $C$ to be nonzero. 
\begin{theorem}
	The condition number $\kappa_{W}(h,x)$ with respect to $\left\| \cdot \right\|_{1}$ admits the upper bound estimate
	\begin{equation}\nonumber
		\kappa_{W}(h,x) \leq \sqrt{n_{o}nn_{h}}\max\left\lbrace \left\| x \right\|_{1}, \left\| h \right\|_{1} \right\rbrace \kappa_{2}\left( A^{T}\left( B^{T}\otimes_{KR}G \right) \right).
	\end{equation}
\end{theorem}
\begin{proof}
	We must get appropriate lower bound estimates for the denominator of $\kappa_{w}(h,x)$ in order to get an upper bound estimate of the
	whole expression.  First, observe that we have from vector norm equivalence
	$\left\| A^{T}\left( B^{T}\otimes_{KR}G \right) \cdot \left( x\otimes h \right) \right\|_{1} \geq \left\| A^{T}\left( B^{T}\otimes_{KR}G \right) \cdot \left( x\otimes h \right) \right\|_{2}$.  Next, we take advantage of the fact that in theory, the three matrices are 
	square, nonsingular Vandermonde matrices.  This allows us to
	unambgiuously construct a chain of inequalities,
	\begin{align*}
		\left\| A^{T}\left( B^{T}\otimes_{KR}G \right) \cdot \left( x\otimes h \right) \right\|_{2} & = \left\| A^{T}\left( B^{T}\otimes_{KR}G \right) \cdot \dfrac{ x\otimes h}{\left\|  x\otimes h \right\|_{2}}  \right\|_{2}\left\|  x\otimes h \right\|_{2}\\
		& \geq \min_{y\in\mathbf{R}^{nn_{h}} \atop \left\| y \right\|_{2}=1} \left\| A^{T}\left( B^{T}\otimes_{KR}G \right) \cdot y \right\|_{2}\left\|  x\otimes h \right\|_{2}\\
		& \geq \sigma_{min}\left( A^{T}\left( B^{T}\otimes_{KR}G \right) \right)\left\|  x\otimes h \right\|_{2}.
	\end{align*}
	Applying this to \eqref{eqn.WCond}
	\begin{equation}\label{eqn.condIneqUpdate1}
		\kappa_{W}(h,x) \leq \dfrac{\left\| A^{T}\left( B^{T}\otimes_{KR}G \right) J_{y}(x,h) \right\|_{1}\cdot \left\| x\otimes h \right\|_{1}}{\sigma_{min}\left( A^{T}\left( B^{T}\otimes_{KR}G \right) \right)\left\|  x\otimes h \right\|_{2}}.
	\end{equation}
	Using vector norm equivalence, we can estimate $\left\| x\otimes h \right\|_{1} \leq \sqrt{nn_{h}}\left\| x\otimes h \right\|_{2}$.  
	We have the induced matrix norm inequality 
	\begin{equation}\nonumber
		\left\| A^{T}\left( B^{T}\otimes_{KR}G \right) J_{y}(x,h) \right\|_{1} \leq \left\| A^{T}\left( B^{T}\otimes_{KR}G \right)\right\|_{1}\left\| J_{y}(x,h) \right\|_{1}.
	\end{equation}
	Writing out $ J_{y}(x,h)$, one observes that $\left\| J_{y}(x,h) \right\|_{1} = \max\left\lbrace \left\| x \right\|_{1}, \left\| h \right\|_{1} \right\rbrace$.
	Furthermore, for $C\in\mathbb{R}^{m_{1}\times m_{2}}$, we have the matrix norm equivalence $\dfrac{1}{\sqrt{m_{1}}}\left\| C\right\|_{1} \leq \left\| C\right\|_{2} \leq \sqrt{m_{2}}\left\| C\right\|_{1} $.
	Applied in this setting yields
	\begin{equation}\nonumber
		\left\| A^{T}\left( B^{T}\otimes_{KR}G \right)\right\|_{1} \leq \sqrt{n_{o}}\left\| A^{T}\left( B^{T}\otimes_{KR}G \right)\right\|_{2} = \sqrt{n_{o}}\,\sigma_{max}\left(  A^{T}\left( B^{T}\otimes_{KR}G \right) \right).		
	\end{equation}
	Substituting into \eqref{eqn.condIneqUpdate1} yields the result after some simplificaton.
\end{proof}
What this demonstrates is that in the one-dimensional case, conditioning of the underlying 
problem, i.e., the relative condition number of the convolution with respect to $\left\| \cdot \right\|_{1}$, has a bound which grows worse exponentially
as the size of the Vandermonde matrices grows. This is not surprising in light of the error analysis shown earlier. 
However, we also see that the 
condition number has an upper bound depending on $\left\| x \right\|_{1}$ and $\left\| h \right\|_{1}$, meaning that we cannot 
rule out that the convolution may exhibit poorer conditioning in the case that $x$ and $h$ are pathologically large.